\documentclass[reqno,oneside,11pt]{amsart}

\calclayout

\usepackage[utf8]{inputenc}
\usepackage[T1]{fontenc}
\usepackage{fourier}
\usepackage{amssymb , amsmath, amsthm}
\usepackage{tikz}
\usetikzlibrary{shapes.geometric}
\usetikzlibrary{calc}
\usepackage{graphicx}
\usepackage{mathtools}
\mathtoolsset{showonlyrefs}

\theoremstyle{definition}
\newtheorem{defi}{Definition}[section]

\newtheorem{remark}[defi]{Remark}
\theoremstyle{plain}
\newtheorem{theorem}[defi]{Theorem}
 \newtheorem{prop}[defi]{Proposition}
\newtheorem{lemma}[defi]{Lemma}
\newtheorem{cor}[defi]{Corollary}

\newtheorem*{conjecture}{Conjecture}

\usepackage{marginnote}

\usepackage[pdftitle={Large Steklov eigenvalues via homogenisation on
  manifolds},
pdfauthor={Alexandre Girouard, Jean Lagacé}, 
pdfborder={0 0 0}
]{hyperref}

\makeatletter

\let\TagsLeftOn\tagsleft@true
\let\TagsLeftOff\tagsleft@false
\makeatother

\newcommand{\B}{\mathbb B}

\newcommand{\R}{\mathbb R}
\newcommand{\N}{\mathbb N}

\newcommand{\T}{\mathbb{T}}
\newcommand{\ukn}{u_k^{(\eps)}}

\newcommand{\de}{\, \mathrm{d}}
\newcommand{\del}{\partial}

\newcommand{\Vol}{\operatorname{Vol}}

\newcommand{\Ukn}{U_k^{(\eps)}}

\newcommand{\CG}{\mathcal{G}}
\newcommand{\CH}{\mathcal{H}}

\newcommand{\CN}{\mathcal{N}}

\newcommand{\BT}{\mathbf T}

\newcommand {\RH}{\mathrm H}
\newcommand {\RL}{\mathrm L}
\newcommand{\RB}{\mathrm B}
\newcommand{\RV}{\mathrm V}
\newcommand{\RBV}{\RB\RV}
\newcommand{\RC}{\mathrm C}
\newcommand{\BS}{\mathbf S}
\newcommand{\BV}{\mathbf V}
\newcommand{\RP}{\mathbb{RP}}

\newcommand{\abs}[1]{\left\lvert #1 \right\rvert}

\newcommand{\set}[1]{\left\{ #1 \right\}}

\newcommand{\norm}[1]{\left\| #1 \right\|}
\newcommand{\bo}\boldsymbol{}
\newcommand{\bigo}[2][]{O_{#1}\left( #2 \right)}
\newcommand{\smallo}[2][]{o_{#1}\left( #2 \right)}

\DeclareMathOperator{\dist}{dist}

\DeclareMathOperator{\inj}{inj}

\DeclareMathOperator{\diam}{diam}

\newcommand{\ceil}[1]{\left\lceil#1\right\rceil}

\renewcommand{\hat}{\widehat}
\renewcommand{\tilde}{\widetilde}
\newcommand{\eps}{\varepsilon}
\renewcommand{\phi}{\varphi}

\renewcommand{\S}{\mathbb S}

\title[Large Steklov eigenvalues]{
  Large Steklov eigenvalues via homogenisation on
  manifolds}

\author{Alexandre Girouard}
\address{D\'epartement de math\'ematiques et de statistique, Pavillon Alexeandre-Vachon, Universit\'e Laval, Qu\'ebec, QC, G1V 0A6, Canada}
\email{alexandre.girouard@mat.ulaval.ca}

\author{Jean Lagac\'e}
\address{Department of Mathematics, University College London, Gower Street, London, WC1E 6BT, United Kingdom}
\email{j.lagace@ucl.ac.uk}

\begin{document}
\begin{abstract}
Using methods in the spirit of deterministic homogenisation theory
 we obtain convergence of the Steklov eigenvalues of a sequence of domains in a Riemannian manifold to weighted Laplace eigenvalues of that manifold. The
domains are obtained by removing small geodesic balls that are asymptotically
densely
uniformly distributed as their radius tends to zero. We use this relationship to
construct manifolds that have large Steklov eigenvalues.
  
In dimension two, and with constant weight equal to $1$, we prove that
Kokarev's upper bound of $8\pi$ for the first nonzero normalised Steklov
eigenvalue on orientable surfaces of genus 0 is saturated. For other topological
types and eigenvalue indices, we also obtain
  lower bounds on the best upper bound for the eigenvalue in terms of Laplace
  maximisers. For the first two eigenvalues, these lower bounds become
  equalities.
  A surprising consequence is the existence of free boundary minimal surfaces
  immersed in the unit ball by first Steklov eigenfunctions and with area
  strictly larger than $2\pi$. This was previously thought to be impossible. 
	We provide numerical evidence that some of the already known examples of
    free boundary minimal surfaces have these
    properties and also exhibit simulations of new free boundary minimal
    surfaces of genus zero in the unit ball with even larger area. 
	We prove that the first nonzero Steklov eigenvalue of all these examples is equal to
    1, as a consequence of their symmetries and topology,
    so that they verify a general conjecture by Fraser and Li.
  
  In dimension three and larger, we prove that the isoperimetric inequality of
  Colbois--El Soufi--Girouard is sharp and implies an upper bound for weighted
  Laplace eigenvalues. We also show that in any manifold
  with a fixed metric, one can construct by varying the weight a domain with
  connected boundary whose
  first nonzero normalised Steklov eigenvalue is arbitrarily large.
\end{abstract}

\maketitle

\section{Introduction and main results}

\subsection{The Laplace and Steklov eigenvalue problems}

Let $(M,g)$ be a smooth, closed connected Riemannian manifold of dimension $d\geq 2$ and let
$\Omega \subset M$ be a domain with smooth boundary $\del \Omega$. Let $\beta\in
\RC^\infty(M)$ be a smooth positive function. We study the weighted Laplace
eigenvalue problem
\begin{equation}
  \label{prob:laplace}
  -\Delta \phi = \lambda\beta \phi \qquad \text{in } M
\end{equation}
and the Steklov eigenvalue problem
\begin{equation}
  \label{prob:steklov}
  \begin{cases}
    \Delta u = 0 & \text{in } \Omega;\\
    \del_\nu u = \sigma u & \text{on } \del \Omega;
  \end{cases}
\end{equation}
where $\Delta$ is the Laplace operator and $\nu$ is the outwards unit
normal. The spectra of the Laplace and Steklov problems are discrete and their
eigenvalues form sequences
\begin{equation}
  0 = \lambda_0 < \lambda_1(M,g,\beta) \le \lambda_2(M,g,\beta) \le \dotso \nearrow
  \infty
\end{equation}
and
\begin{equation}
  0 = \sigma_0 < \sigma_1(\Omega,g) \le \sigma_2(\Omega,g) \le  \dotso
  \nearrow \infty
\end{equation}
accumulating only at infinity. Problem~\eqref{prob:laplace} is a staple of
geometric spectral theory, see e.g.~\cite{bergergauduchonmazet,chaveleigriem}.
The eigenvalues $\lambda_k(M,g,\beta)$ correspond to natural frequencies of a
membrane that is non-homogeneous when $\beta$ is not constant. It has recently
been studied by Colbois--El Soufi~\cite{ColboisElSoufi2019} and Colbois--El
Soufi--Savo~\cite{CeSS} in the Riemannian setting.
Problem~\eqref{prob:steklov} is a classical problem originating in mathematical physics \cite{stek1902}
and which has received growing attention in the last few
years. Its eigenvalues are those of the Dirichlet-to-Neumann operator, which maps a
function $f$ on $\del \Omega$ to the normal
derivative on the boundary of its harmonic extension.  See~\cite{gpsurvey} for a survey.

Our main theorem states that for any positive $\beta \in \RC^\infty(M)$, 
Problem~\eqref{prob:laplace} may be realized as a limit of Problem~\eqref{prob:steklov} defined
on carefully constructed domains $\Omega^\eps \subset M$. Denote by $\de
\mu_g$ the Lebesgue measure on $M$ and for every domain $\Omega \subset M$ by $\de A_{\del \Omega}$ the measure on
$M$ defined by integration against the Hausdorff measure on $\del \Omega$. 
\begin{theorem}
  \label{thm:approx}
  Let $(M,g)$ be a closed Riemannian manifold, and  $\beta \in \RC^\infty(M)$ positive. There is
  a sequence of domains $\Omega^\eps \subset M$
  such  that $\de A_{\del \Omega^\eps}$ converges weak-$*$ to $\beta \de \mu_g$  and
  \begin{equation}
  \sigma_k(\Omega^\eps,g)\xrightarrow{\eps \to 0}\lambda_k(M,g,\beta).
\end{equation}
\end{theorem}

The proof of Theorem~\ref{thm:approx} is in the spirit of 
Girouard--Henrot--Lagac\'e~\cite{GHL} where Neumann eigenvalues
of a domain in Euclidean space are related to Steklov eigenvalues of subdomains through
periodic homogenisation by obstacles.

Homogenisation theory is a branch of applied mathematics that is interested in
the study of PDEs and variational problems in the presence of structures at
many different scales; in the presence of two scales they are usually referred to
as the \emph{macrostructure} and \emph{microstructure}. The methods are usually
divided in two general categories: deterministic (or periodic), and
stochastic. 
The effectiveness of
homogenisation in shape optimisation, see for example the Allaire's influencial monograph~\cite{allaire} and the references therein, leads one to believe that
it should also be useful elsewhere in geometric analysis. 

The main obstacle to the application of deterministic homogenisation theory in
the Riemannian setting is that most Riemannian manifolds do not exhibit any form
of periodic structure. It is therefore not surprising that
homogenisation theory in this setting has either been 
applied when an underlying manifold exhibits a periodic-like structure, see e.g.
the work of Boutet de Monvel--Khruslov~\cite{BdMK} and
Contreras--Iturriaga--Siconolfi~\cite{CIS}, or used the periodic structure of an
ambient space in which a manifold is embedded, see Braides--Cancedda--Chiad\`o
Piat~\cite{BCC} or relied on an imposed periodic structure in predetermined charts, see
Dobbersch\"utz--B\"ohm~\cite{DB}. Our approach is distinct in that it is entirely
intrinsic and does not require a periodic structure at any stage.

We note that stochastic homogenisation has been used in geometric contexts, see e.g. the
recent paper by Li
\cite{xmli}. Chavel--Feldman~\cite{chavelfeldman1,chavelfeldman2} also studied
the effect on the spectrum of the Laplacian of removing a large, but fixed,
number of small geodesic balls on which Dirichlet boundary conditions are
imposed. However, no consideration was given to the distribution of those
geodesic balls, nor to asymptotic behaviour joint in the number of balls removed
and their size.

\begin{remark}
  It is natural to expect that the Steklov eigenvalues of a domain $\Omega$ with
  smooth boundary would be related to the
eigenvalues $\lambda_k(\partial\Omega)$ of the Laplace operator \emph{of its
boundary}, since the Dirichlet-to-Neumann map
is an elliptic pseudo-differential operator that has the same principal symbol
as the square root of the Laplace operator on $\del \Omega$, see \cite[Section
7.11]{Taylor}. Indeed, upper bounds for $\sigma_k(\Omega)$ in terms of
$\lambda_k(\partial\Omega)$ have been obtained by Wang--Xia~\cite{wangxia2} for
$k=1$ and by Karpukhin~\cite{Karpukhin2017} for higher eigenvalues.
Quantitative estimate for $|\sigma_k(\Omega)-\sqrt{\lambda_k(\partial\Omega)}|$
have been obtained by Provenzano--Stubbe~\cite{ProvStub2019} for 
domains in Euclidean space and by Xiong~\cite{Xiong2018} and
Colbois--Girouard--Hassannezhad \cite{ColbGirHas2020} in the Riemannian setting.
The eigenvalues of various other spectral problems have also been compared with
Steklov eigenvalues. See the work of Kuttler--Sigillito~\cite{kuttsigi1} and
Hassannezhad--Siffert~\cite{HassannezhadSiffert}.

A different type of relationship was studied in
Lamberti--Provenzano~\cite{LambertiProvenzano2015}, where it is proved that the
Steklov eigenvalues of a domain $\Omega\subset\R^d$ can be obtained as
appropriate limits of non-homogeneous Neumann eigenvalues with the mass
concentrated at the boundary of $\Omega$.

\end{remark}

\subsection{Isoperimetric inequalities}
Theorem~\ref{thm:approx} has several applications to the study of isoperimetric
inequalities for Steklov eigenvalues. These are most naturally stated in terms
of the scale invariant eigenvalues
\begin{equation}
  \label{eq:isolamb}
  \Lambda_k(M,g) := \Vol_g(M)^{2/d} \lambda_k(M,g,1)
\end{equation}
and
\begin{equation}
  \label{eq:isostek}
  \Sigma_k(\Omega,g) := \CH^{d-1}(\del \Omega)^{1/(d-1)} \sigma_k(\Omega,g),
\end{equation}
where $\Vol_g(M)$ is the volume of $M$ and $\CH^{d-1}(\del \Omega)$ is the $(d-1)$-Hausdorff
measure of the boundary $\del \Omega$.
It is natural to ask for upper bounds on the functionals \eqref{eq:isolamb} and
\eqref{eq:isostek}, and as such to define
\begin{equation}
  \Lambda_k^*(M) := \sup_{g\in\CG(M)} \Lambda_k(M,g)
\end{equation}
and
\begin{equation}
  \Sigma_k^*(M) := \sup_{\Omega \subset M} \sup_{g \in \CG(\overline \Omega)}
  \Sigma_k(\Omega,g)
\end{equation}
where for any manifold, with or without boundary, $\CG(M)$ is the set of all Riemannian metrics on $M$. Spectral isoperimetric inequalities often have a
wildly different behaviour in dimension two than in dimension at least three, as
exhibited in the work of Colbois--Dodziuk~\cite{cd} and Korevaar~\cite{kvr}. As
such, we study these cases separately.

\subsubsection{Isoperimetric inequalities in dimension two}
From Colbois--El Soufi--Girouard~\cite{ceg2} it is known that $\Sigma_k^*(M)$ is
finite for each surface. The next result provides an effective lower bound.
\begin{theorem}
  \label{thm:main}
  For every $k \in \N$ and every smooth, closed, connected surface $M$,
  \begin{equation}
    \Sigma_k^*(M) \ge \Lambda_k^*(M).
  \end{equation}
\end{theorem}
This should be compared with \cite[Theorem 9]{GHL} where a similar inequality
was proved, relating Steklov and Neumann eigenvalues of a domain in Euclidean
space.
The storied study of $\Lambda_k^*$ for various $k$ and smooth surfaces $M$ of
different topologies
yields explicit lower bounds for $\Sigma_k^*$, which we 
record in Section \ref{sec:fbms}. 
Kokarev~\cite[Theorem $A_1$, Example 1.3]{kok}
proved that $\Sigma_1^*(\S^2) \le 8\pi$.
Theorem
\ref{thm:main} and the known value $\lambda_1(\S^2,g_{0})=2$ for the round sphere,
shows that Kokarev's bound is sharp.
\begin{cor}
  \label{cor:sphere}
  The following equality holds:
  \begin{equation}
    \Sigma_1^*(\S^2) = 8 \pi.
  \end{equation}
\end{cor}

\begin{remark}
  \label{rem:counter}
  Both Corollary \ref{cor:sphere} and, further along, Theorem \ref{thm:tetrahedral} along with
  \eqref{eq:boundbeaten} are in contradiction with parts of \cite[Theorems 8.2]{fraschoen2}, where the bound
  \begin{equation}
  \Sigma_1^*(\S^2) \le 4 \pi
\end{equation}
  is given. Further discussion and related results are delayed to Appendix
  \ref{sec:comment}.
\end{remark}

Very recent work of Karpukhin--Stern \cite[Theorem 5.2]{KS} in fact shows that for all
  surfaces $M$, and for $j \in \set{1,2}$
  \begin{equation}
    \Sigma_j^*(M) \le  \Lambda_j^*(M),
  \end{equation}
  using methods from min-max theory of harmonic maps. In combination with Theorem
  \ref{thm:main}, we obtain the following result, also presented as
  \cite[Proposition 5.9]{KS}, which extends Corollary
\ref{cor:sphere}.
  \begin{cor}
    \label{cor:equality}
    For all closed surfaces $M$, the following equalities hold
    \begin{equation}
      \Sigma_1^*(M) = \Lambda_1^*(M)
    \end{equation}
    and
    \begin{equation}
      \Sigma_2^*(M) = \Lambda_2^*(M).
    \end{equation}
  \end{cor}
This leads naturally to the following conjecture.
\begin{conjecture}
  For all closed surfaces $M$ and all $k \in \N$,
  \begin{equation}
    \Sigma_k^*(M)=\Lambda_k^*(M).
\end{equation}
\end{conjecture}

\subsubsection{Isoperimetric inequalities in dimension at least three}
For $d\geq 3$, it follows from the work of Colbois--Dodziuk~\cite{cd} that
$\Lambda_1^*(M)=+\infty$. Together with Theorem~\ref{thm:approx} this gives
$\Sigma_1^*(M)=+\infty$. Using the extra freedom provided by the weight $\beta$, we arrive
at more precise statements, starting with the following corollary to Theorem \ref{thm:approx}.

\begin{cor}
  \label{cor:betad}
  Let $(M,g)$ be a Riemannian manifold of dimension $d \ge 2$. For
 constant density $\beta >0$, the domains
  $\Omega^\eps\subset M$ obtained in Theorem \ref{thm:approx} satisfy
  \begin{equation}
    \Sigma_k(\Omega^\eps,g) \xrightarrow{\eps \to 0}
    \beta^{\frac{2-d}{d-1}}\Vol_g(M)^{\frac{2-d}{d(d-1)}}\Lambda_k(M,g).
  \end{equation}
\end{cor}

This is a direct consequence of Theorem~\ref{thm:approx} since for constant
density $\beta>0$ one has
$$\lambda_k(M,g,\beta)=\frac{1}{\beta}\lambda_k(M,g).$$
By removing thin tubes joining boundary components of a domain $\Omega\subset M$,
Fraser--Schoen~\cite{FraserSchoenHigherd} proved that in dimension $d \ge 3$,
there is a family of domains $\Omega_\delta\subset\Omega$, with 
connected boundary and such that $\abs{\Sigma_1(\Omega) -
\Sigma_1(\Omega_\delta)} < \delta$.
In combination
with the previous corollary, this leads to the following result.
\begin{cor}\label{cor:large}
  Let $(M,g)$ be a Riemannian manifold of dimension $d\geq 3$. Then there exists
  a sequence of domains $\Omega^\eps\subset M$ with connected boundary such that
  \begin{equation}
    \lim_{\eps\to 0}\Sigma_1(\Omega^\eps,g)=+\infty.
  \end{equation}
\end{cor}
In recent years several constructions of manifolds with large normalised Steklov
eigenvalue $\Sigma_1$ have been proposed. Colbois--Girouard~\cite{colbgir1} and
Binoy~\cite{ColbGirBinoy} constructed a sequence $\Omega_n$ of compact surfaces
with connected boundary such that $\Sigma_1(\Omega_n)\to\infty$.
Cianci--Girouard~\cite{CianciGirouard}  proved that some manifolds $M$ of
dimension $d\geq 4$ carry Riemannian metrics that are prescribed on $\partial M$
with uniformly bounded volume and arbitrarily large first Steklov eigenvalue
$\sigma_1$. 
Corollary~\ref{cor:large} provides a
new outlook on this question.

\subsubsection{Transferring bounds for Steklov eigenvalues to bounds for Laplace eigenvalues}

If $(M,g)$ is conformally equivalent to a Riemannian manifold with non-negative Ricci curvature, it follows from Colbois--Girouard--El Soufi \cite{ceg2} that  for each domain $\Omega \subset M$ with smooth boundary, and for each $k \ge 1$,
\begin{equation}
  \label{eq:isohighd}
  \sigma_k(\Omega)
\leq C_d\frac{\Vol_g(\Omega)^{\frac{d-2}{d}}}{\CH^{d-1}(\del
  \Omega)}k^{2/d}.
\end{equation}
Using the domains $\Omega^\eps$ from Theorem~\ref{thm:approx} and taking the limit as $\eps\to 0$ leads to the following.
\begin{cor}
  Let $(M,g)$ be a closed manifold with $g$ conformally equivalent to a metric with nonnegative Ricci curvature. For each $\beta\in C^\infty(M)$ positive,
  \begin{equation}\label{ineq:gny}
    \lambda_k(M,g,\beta)\int_M\beta\,d\mu_g\leq C_d\Vol_g(M)^{\frac{d-2}{d}}k^{2/d}.
  \end{equation}
\end{cor}
  This is a special case of an inequality that was proved in Grigor'yan--Netrusov--Yau~\cite[Theorem 5.9]{GNY2004}.

  \begin{cor}\label{cor:exponent}
    The exponent\, $2/d$ cannot be improved  in \eqref{eq:isohighd}, and 
    the exponents on $\Vol_g(\Omega)$ and $\CH^{d-1}(\del \Omega)$ cannot be replaced by any other exponents.
  \end{cor}
  Indeed for $\beta>0$ constant, inequality~\eqref{ineq:gny} becomes
$\lambda_k(M,g)\Vol_g(M)^{\frac{2}{d}}\leq C_dk^{2/d}$, where the exponent
carries over from \eqref{eq:isohighd}. That it cannot be improved follows from
the Weyl Law. Now, changing the exponent of $\CH^{d-1}(\del\Omega)$
in~\eqref{eq:isohighd} would yield an inequality with a non-trivial exponent for
$\beta$, while changing the exponent of $\Vol_g(\Omega)$ would lead to an
inequality similar to~\eqref{ineq:gny}, but not invariant under scaling of the
Riemannian metric.

\begin{remark}
  Corollary~\ref{cor:exponent} improves uppon~\cite[Remark 1.4]{ceg2}, where it was already observed that the exponent $2/d$ could not be replaced by $1/(d-1)$ in inequality~\eqref{eq:isohighd}. Note also that for an Euclidean domain $\Omega\subset\R^d$, it follows from the isoperimetric inequality and~\eqref{eq:isohighd} that $\sigma_k(\Omega)|\partial\Omega|^{1/(d-1)}\leq Ck^{2/d}$. Deciding if the exponent $2/d$ can be improved in this inequality is still an open problem, which was proposed as~\cite[Open problem 5]{gpsurvey}.
\end{remark}

\subsection{Free boundary minimal surfaces}
In dimension $d=2$, the striking connection between
the Steklov eigenvalue problem and free boundary minimal submanifolds in the
unit ball was revealed by Fraser and Schoen in \cite{fraschoen,fraschoen3,fraschoen2}.
\begin{defi}[{cf. \cite[Theorem 2.2]{Li19}}]
  \label{defi:fbms}
  For $m \ge 3$, let $\mathbb{B}^m$ be the $m$-dimensional Euclidean unit
  ball and let $\Omega\subset\mathbb{B}^m$ be a $k$-dimensional submanifold with
  boundary $\partial\Omega=\overline{\Omega}\cap\partial\mathbb{B}^m$. 
  We say that $\Omega$ is a \emph{free boundary minimal submanifold} in $\mathbb{B}^m$ if one of the following equivalent conditions hold. 
\begin{enumerate}
  \item\label{fbms-1} $\Omega$ it is a critical point for the area functional among all $k$-dimensional submanifolds of $\mathbb{B}^m$ with boundary on $\partial \mathbb{B}^m$. 
  \item\label{fbms-2} $\Omega$ has vanishing mean curvature and meets $\partial\mathbb{B}^m$ orthogonally. 
  \item\label{fbms-3} The coordinate functions $x^1,\dotsc,x^{m}$ restricted to $\Omega$ are solutions to the Steklov eigenvalue problem \eqref{prob:steklov} with eigenvalue $\sigma=1$. 
\end{enumerate}
\end{defi}
Conditions \eqref{fbms-1} and \eqref{fbms-2} can be used to generalise Definition
\ref{defi:fbms} to arbitrary background manifolds in place of $ \mathbb{B}^m$,
but the equivalence of condition \eqref{fbms-3} is a special property of the
Euclidean unit ball. 
Conversely, in \cite[Proposition 5.2]{fraschoen2}, it was proven for surfaces
that maximal metrics $g$ on $\Omega$ for
$\Sigma_1$ have first Steklov eigenfunctions which realise an isometric immersion of $\Omega$ as a free boundary minimal
surface inside the unit ball $\B^m$. 
It is conjectured by Fraser and Li \cite[Conjecture 3.3]{FraserLi} that $\sigma=1$ is actually
equal to the \emph{first} nonzero Steklov eigenvalue $\sigma_1(\Omega)$ for any
given compact, properly embedded free boundary minimal hypersurface $\Omega$ in
the unit ball. 
Even in the case $m=3$ it is a
challenging problem to construct free boundary minimal surfaces with a given
topology. The first nontrivial examples (apart from the equatorial disk and the
critical catenoid) were found by Fraser and Schoen \cite{fraschoen2}. 
Their surfaces have genus 0 and an arbitrary number of boundary components. 
An independent construction of free boundary minimal surfaces with genus
$\gamma\in\{0,1\}$ and any sufficiently large number $b$ of boundary
components was given by Folha--Pacard--Zolotareva \cite{Folha2017}. 
The sequence of surfaces converges as $b\to\infty$ to the equatorial disk with multiplicity two. 
McGrath \cite[Corollary 2]{McGrath2018} proved that these surfaces indeed have
the property that $\sigma_1=1$ as conjectured by Fraser and Li.

Let us now mention a few other constructions for which it is an open
problem whether $\sigma_1 = 1$. Free boundary minimal surfaces with high genus
were constructed by Kapouleas--Li \cite{KapLi17} and Kapouleas--Wiygul
\cite{KapouleasWiygul2017} using desingularisation methods.  
The equivariant min-max theory developed by Ketover
\cite{Ketover2016equiv,Ketover2016fb} allowed the construction of free boundary
minimal surfaces of arbitrary genus with dihedral symmetry and of genus~0 with
symmetry group associated to one of the platonic solids. 
If their genus is sufficiently high, Ketover's surfaces have three boundary components. 
More recently, Carlotto--Franz--Schulz \cite{Carlotto2020}
constructed free boundary minimal surfaces with dihedral symmetry, arbitrary
genus and connected boundary.  

For certain free boundary minimal surfaces which are invariant under the action
of the
symmetry group associated to one of the platonic solids (see \cite[Theorem
6.1]{Ketover2016fb}) we confirm Fraser and Li's conjecture about the first
Steklov eigenvalue in the following theorem based on the work of \mbox{McGrath}
\cite{McGrath2018}.

\begin{theorem}\label{thm:tetrahedral}
Let\, $\Omega\subset\mathbb{B}^3$ be an embedded free boundary
minimal surface of genus $0$. 
If\, $\Omega$ has 
tetrahedral symmetry and $b=4$ boundary components or 
octahedral symmetry and $b\in\{6,8\}$ boundary components or 
icosahedral symmetry and $b\in\{12,20,32\}$ boundary components, 
then $\sigma_1(\Omega)=1$. 
\end{theorem} 

\begin{remark}
Ketover's result \cite[Theorem 6.1]{Ketover2016fb} states the existence of free boundary minimal surfaces 
with tetrahedral symmetry and $b=4$ boundary components, with octahedral
symmetry and $b=6$ boundary components and with 
icosahedral symmetry and $b=12$ boundary components.   
We conjecture that free boundary minimal surfaces with $b\in\{8,20,32\}$
boundary components and corresponding symmetries as stated in
Theorem~\ref{thm:tetrahedral} exist as well. 
In fact, we visualise all mentioned cases in Figures \ref{fig:tetrahedral}, \ref{fig:octahedral} and 
\ref{fig:icosahedral}. 
The simulation is based on Brakke's surface evolver \cite{Brakke1992} which we
use to approximate free boundary minimal disks $D$ inside a four-sided wedge as
shown on the right of Figure \ref{fig:tetrahedral}. 
If the wedge is chosen suitably such that it forms a fundamental domain for the
action of the symmetry group of one of the platonic solids (see Definition
\ref{def:fundamentaldomain}), then repeated reflection of $D$ leads to an
approximation of a free boundary minimal surface in the unit ball. 

The simulations allow approximations for $\Sigma_1$. 
Indeed, in Table \ref{tab:areas} we numerically compute the area of each surface
shown in Figures \ref{fig:tetrahedral}, \ref{fig:octahedral} and
\ref{fig:icosahedral} using the surface evolver. 
To increase accuracy, the area has been computed using a much finer triangulation than
the one used to render the images.   
Since any free boundary minimal surface $\Omega\subset\mathbb{B}^3$ has boundary
length equal to twice its area (see \cite[Proposition~2.4]{Li19}) and
since symmetries and topology imply $\sigma_1(\Omega)=1$ by Theorem~\ref{thm:tetrahedral}, we observe in each case 
\begin{equation}\label{eq:boundbeaten}
\Sigma_1(\Omega)=\CH^1(\partial\Omega)\,\sigma_1(\Omega)>4\pi. 
\end{equation}
\end{remark}

\begin{table}\centering
\begin{tabular}{l|c|c|c}
\bfseries symmetry & \bfseries  boundary components & \bfseries  area & $\Sigma_1(\Omega)$ \\\hline
tetrahedral   &  4 & $2.1752\,\pi$ & $4.3505\,\pi$  \\       
octahedral    &  6 & $2.4549\,\pi$ & $4.9099\,\pi$  \\       
octahedral    &  8 & $2.6141\,\pi$ & $5.2282\,\pi$  \\       
icosahedral   & 12 & $2.8757\,\pi$ & $5.7514\,\pi$  \\       
icosahedral   & 20 & $3.1149\,\pi$ & $6.2299\,\pi$  \\       
icosahedral   & 32 & $3.3444\,\pi$ & $6.6888\,\pi$  \\\hline 
\end{tabular}
\medskip
\caption{Areas and scale invariant eigenvalues of the surfaces shown in Figures \ref{fig:tetrahedral}, \ref{fig:octahedral} and \ref{fig:icosahedral}.}
\label{tab:areas}
\end{table}

We emphasise that we do \emph{not} answer the question whether or not any of the
free boundary minimal surfaces discussed in Theorem \ref{thm:tetrahedral}
respectively Table \ref{tab:areas} are maximisers for $\Sigma_1$ in the class of
surfaces with the same topology.
\begin{figure}%
\includegraphics[width=0.47\textwidth]{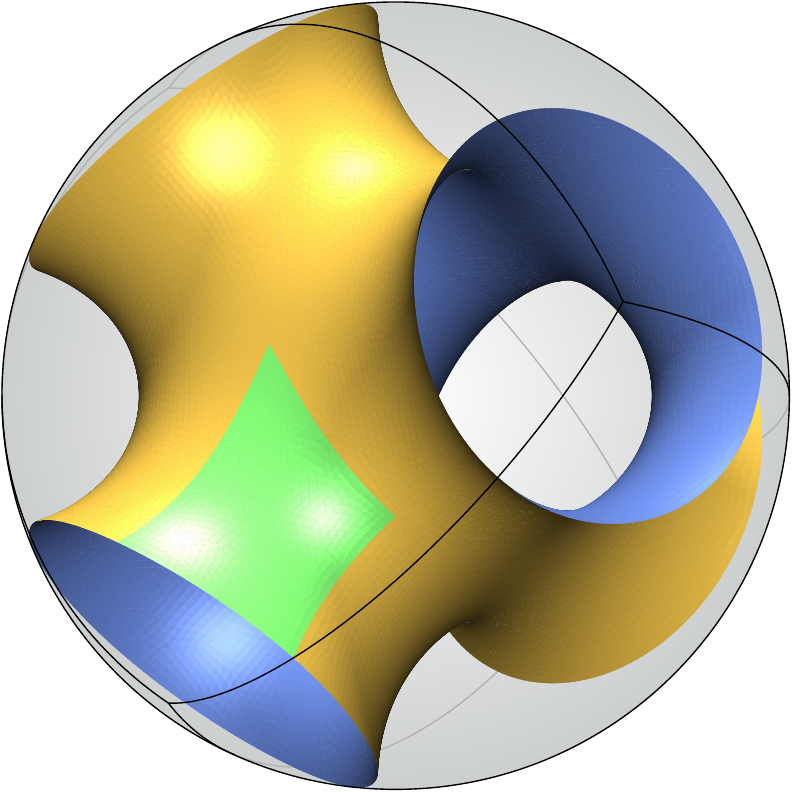}\hfill
\includegraphics[width=0.47\textwidth]{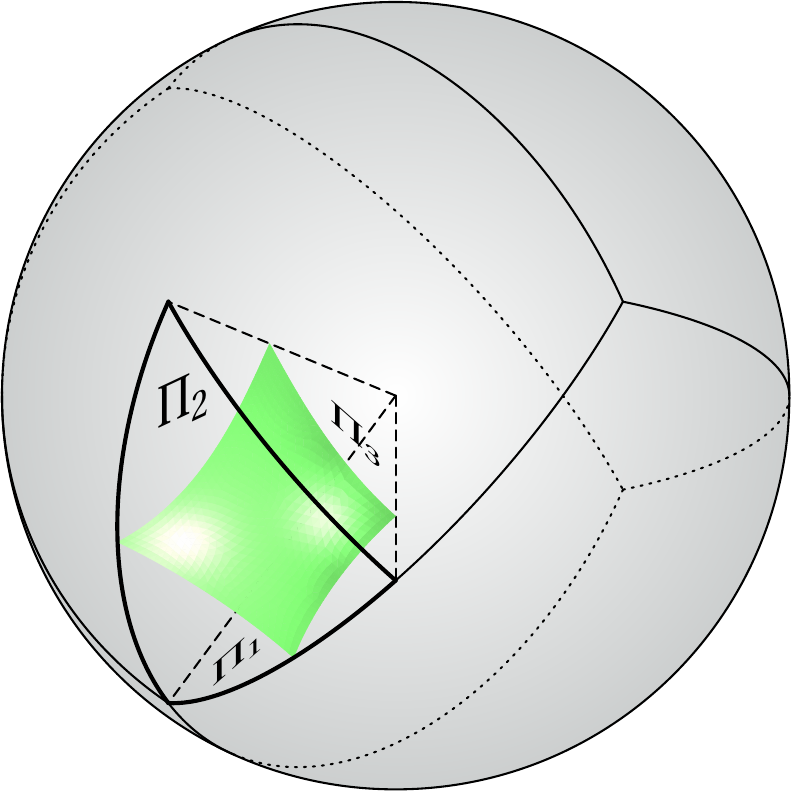}%
\caption{Free boundary minimal surface of genus 0 with tetrahedral symmetry and $4$ boundary components and its fundamental domain being a free boundary minimal disk inside a four-sided wedge.}%
\label{fig:tetrahedral}%
\end{figure}

\begin{figure}%
\includegraphics[width=0.47\textwidth]{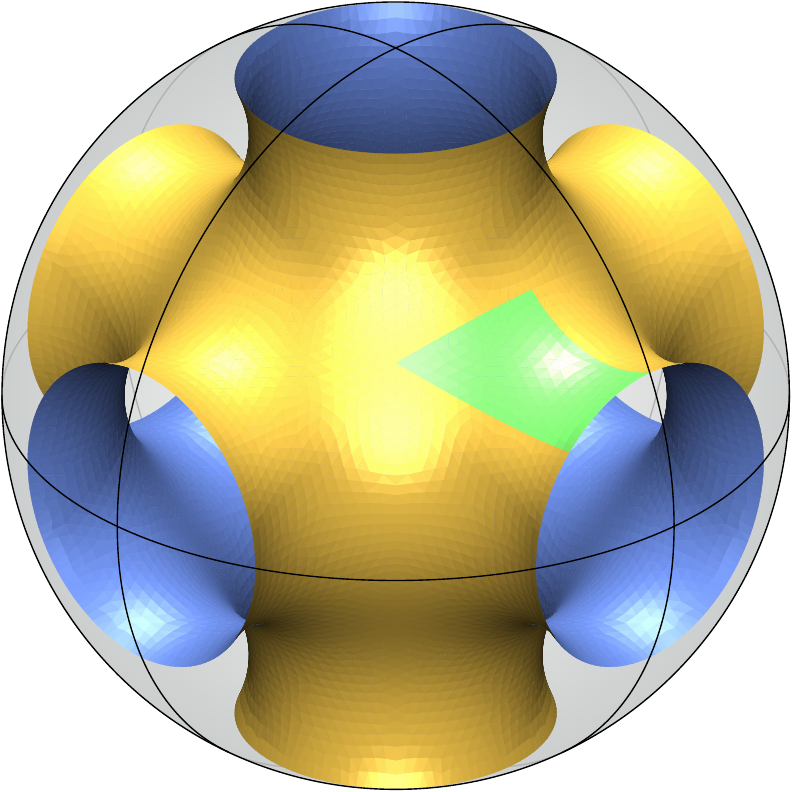}\hfill
\includegraphics[width=0.47\textwidth]{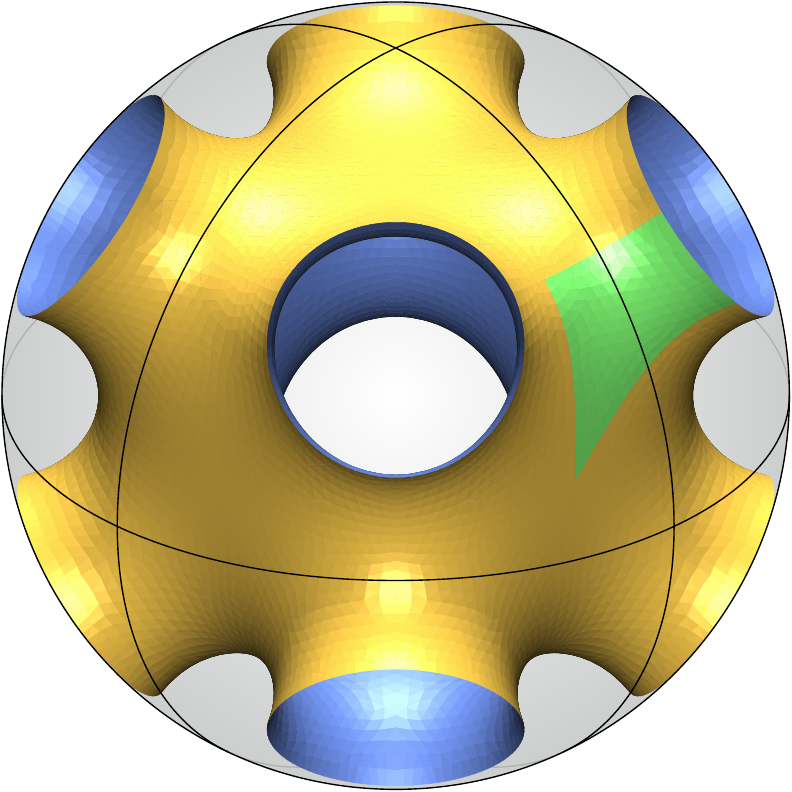}%
\caption{Free boundary minimal surfaces of genus $0$ with octahedral symmetry and $6$ or $8$ boundary components.}%
\label{fig:octahedral}%
\end{figure}

\begin{figure}\centering
\includegraphics[width=0.47\textwidth]{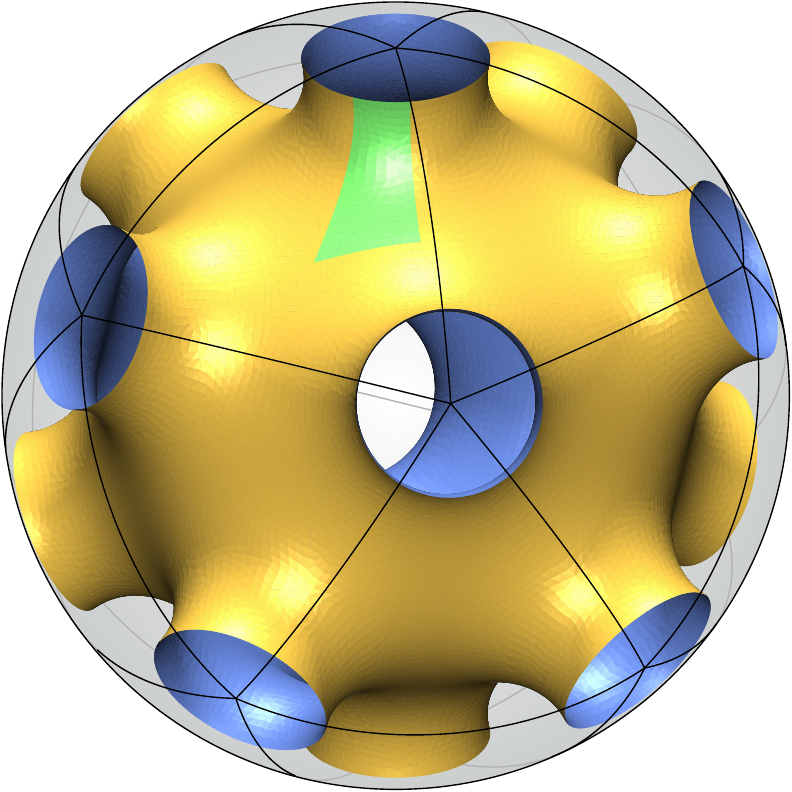}\hfill
\includegraphics[width=0.47\textwidth]{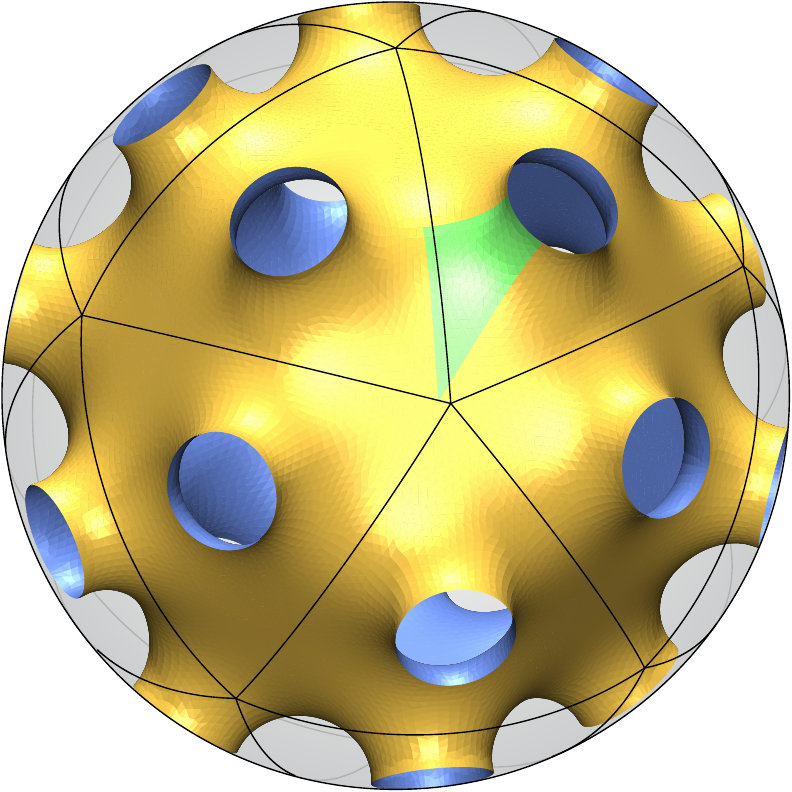}%
\\
\includegraphics[width=0.47\textwidth]{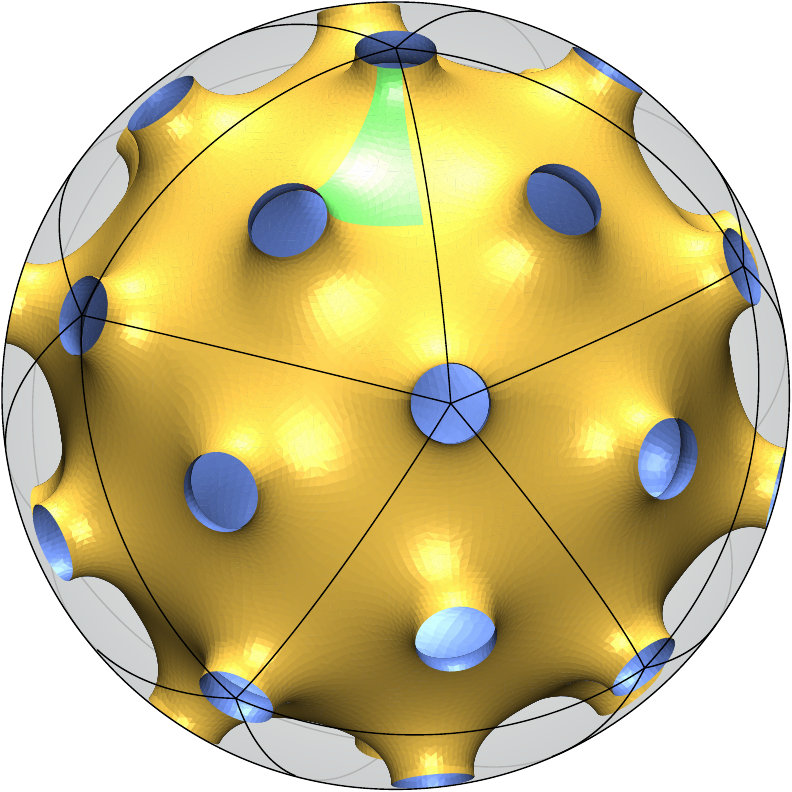}
\caption{Free boundary minimal surfaces of genus $0$ with icosahedral symmetry and $12$, $20$ or $32$ boundary components.}%
\label{fig:icosahedral}%
\end{figure}

\begin{remark}
For Laplace
eigenvalues, the eigenfunctions of a critical metric $g$ on $M$ for
$\Lambda_1$ realise an isometric immersion of $M$ as a minimal surface in the
sphere $\S^m$ for some $m$, see Nadirashvili \cite{nad96}.
\end{remark}

\subsection*{Plan of the paper}

In Section \ref{sec:homodef}, we describe precisely the homogenisation construction in the
Riemannian setting. Theorem \ref{thm:homo} is a restatement of Theorem \ref{thm:approx} in terms
of the explicit sequence of domains for which the normalised Steklov eigenvalues
converge to the weighted Laplace eigenvalues.

In Section \ref{sec:ineq} we prove various technical inequalities that will be used in the
later stages. Some of these inequalities are known for domains in flat
space and we extend their proofs to the Riemannian setting. We first need to
control the norm of the traces $\gamma^\eps :
\RH^1(\Omega^\eps) \to \RL^2(\del\Omega^\eps)$ and $\tau^\eps :
\RBV(\Omega^\eps) \to \RL^1(\del \Omega^\eps)$ uniformly in the parameter
$\eps$. We also need to bound uniformly the norm of the harmonic extension
operator from $\RH^1(\Omega^\eps)$ to $M$, and to have a uniform
Poincaré--Wirtinger inequality for
some topological perturbations of geodesically convex
subsets of $M$. We point out that
the usual sufficient conditions in term of conditions on tubular neighbourhoods
of the boundary and inner cone conditions are not satisfied in our case,
nevertheless we can use the structure of the problem to find the relevant bounds.

In Section \ref{sec:stekprop}, we prove boundedness properties for the
Steklov eigenvalues and eigenfunctions of the domains $\Omega^\eps$. More
previsely, we prove that for every fixed $k$, $\sigma_k^\eps$ is bounded in
$\eps$, and that the $\RL^\infty$ norm of $\ukn$ is also bounded uniformly.

Section \ref{sec:homo} is dedicated to the proof of Theorem \ref{thm:homo}. The
proof proceeds in three main steps. The first one is to show that for every $k$,
the eigenvalues
$\sigma_k(\Omega^\eps)$ are bounded as $\eps \to 0$ as well as to show that the families
of harmonic extensions $\Ukn$ are bounded in $\RH^1(M)$. This gives us the
existence along a subsequence of a limit $\sigma_k^\eps \to \lambda$ and of a
$\RH^1(M)$ weak limit $\Ukn \to \phi$. The second step consists in studying the
weak formulations to show that the pair $(\phi,\lambda)$ is a solution to Problem
\eqref{prob:laplace}. In the last step, we show that there is no mass lost in the
process, and therefore that indeed $\lambda = \lambda_k(M,g,\beta)$.

In Section \ref{sec:isoineq}, we prove the isoperimetric inequality stated in
Theorem \ref{thm:main} and give as a corollary explicit lower bounds on the
maximiser for Steklov eigenvalues in terms of known bounds for Laplace
eigenvalues. 

Finally, in Section \ref{sec:fbms}, we provide a proof of Theorem
\ref{thm:tetrahedral}. This proof uses symmetries of the free boundary minimal
surfaces, and properties of the
nodal sets of first eigenfunctions.

\subsection*{Acknowledgements}
We want to thank A. Fraser and R. Schoen for very useful discussions that led to
better clarity in Appendix \ref{sec:comment}. 
This project stemmed from a reading group at UCL, as such the authors want to
thank in particular C. Belletini, F. Hiesmayr, M. Levitin, L. Parnovski, and F.
Schulze for useful discussions and remarks as the project started, as well as I.
Polterovich for raising the question leading to this paper. We would also
like to thank H. Matthiesen and R. Petrides for very useful
discussions about the state of the art concerning maximisers of
normalised Steklov eigenvalues and free boundary minimal surfaces.
We are grateful to M. Schulz for generously letting them use the
pictures from his website~\cite{MarioWeb}. Finaly, the authors would like to thank M. Karpukhin for very useful comments in the final stages of writing this paper.

AG acknowledges the support of NSERC.
The research of JL was supported by EPSRC grant EP/P024793/1 and the NSERC
Postdoctoral Fellowship. 

\section{The homogenisation construction}

\label{sec:homodef}

\subsection{Notation}

From this section on, we denote by $c$ and $C$ positive constants that may depend
only on the manifold $(M,g)$, the dimension, and the positive smooth function
$\beta\in C^\infty(M)$. Similarly, the homogenisation construction depends on a
parameter $\eps > 0$ which must be chosen smaller than $\eps_0 > 0$, a value
also depending only on $(M,g)$, the dimension, and $\beta$. 
The precise values of $c, C$ and $\eps_0$ may
change from line to line, but changes occur only a finite number of times so that at
the end $0 < \eps_0, c, C < \infty$.

We will reserve the letters $\phi, \lambda$ for general eigenfunctions and eigenvalues
of Problem~\eqref{prob:laplace}, and $\phi_k$ and $\lambda_k$ representing
specifically the $k$th ones. Similarly, we reserve $u^{(\eps)}$ and
$\sigma^{(\eps)}$ for Steklov eigenvalues of the sequence of domains
$\Omega^\eps$. We drop in this notation any specific reference to $M$, to the
metric $g$ and to the weight $\beta$ as they are kept fixed. We assume that
eigenfunctions $\ukn$ and $\phi_k$ are orthonormal, with respect to
$\RL^2(\del\Omega^\eps)$ and $\RL^2(M,\beta \de \mu_g)$ respectively.
We make use of various asymptotic notation.
\begin{itemize}
  \item Indiscriminately, writing $f = \bigo g$ or $f \ll g$ means that there
    exists $C, \eps_0 > 0$ such that $\abs{f(x)} \le C g(x)$ for all $0 < x <
    \eps_0$. 
  \item Writing $f = \smallo g$ means that $\frac{f}{g} \to 0$ as $\eps \to 0$.
  \item Writing $f \asymp g$ means that both $f \ll g$ and $g \ll f$.
  \item Indices in the asymptotic notation (e.g. $f = O_M(g)$ or $f \ll_k g$)
    means that the implicit constants, the range of validity or the rate of convergence to $0$ for
      $o$ depends only on those quantities. We use $M$ as an index to represent
      dependence both on the manifold $M$ and on the metric $g$. 
\end{itemize}

\subsection{Geodesic polar coordinates}

Some of our proofs are formulated using geodesic polar coordinates,
so let us recall their construction, see \cite[Chapter XII.8]{chaveleigriem}.
For a point $p \in M$ and $\delta <
\inj(M)$, the exponential map is a diffeomorphism from the ball of radius
$\delta$ in $T_pM$ to the geodesic ball $B_{\delta}(p) \subset M$. In
$B_\delta(p)$, we use the polar coordinates $(\rho,\theta)$, where $\rho$ is the
geodesic distance from $p$ and $\theta$ is a unit tangent vector in $T_pM$. 

We recall that in those coordinates, the metric reads
\begin{equation}
  \label{eq:metricpolar}
  g(\rho,\theta) = \de \rho^2 + \rho^2(1 + h(\rho,\theta))g_{\S^{d-1}},
\end{equation}
where
\begin{equation}
  \label{eq:perturbpolar}
  \norm{h(\rho,\theta)}_{\RC^1(B_\delta(p))} = O_M(\delta).
\end{equation}
We record as well that the volume element can be written in
 these coordinates as
 \begin{equation}
   \label{eq:areageoball}
   \de V = \rho^{d-1}\left(1 + \bigo[M]{\delta^2}\right) \de \rho \de A_{\S^{d-1}}
 \end{equation}
 and for any geodesic sphere of radius $r \le \delta$, its
 area element is of the form
 \begin{equation}
   \label{eq:boundarygeoball}
   \de A = r^{d-1}\left(1 + \bigo[M]{r^2}\right) \de A_{\S^{d-1}}.
 \end{equation}
 Compactness of $M$ ensures that the implicit constants in
 \eqref{eq:perturbpolar}, \eqref{eq:areageoball} and \eqref{eq:boundarygeoball} 
 can be chosen independently of $p$.

 \subsection{Homogenisation by obstacles}

 For every $\eps > 0$, let $\BS^\eps$ be a maximal $\eps$-separated subset of
 $M$, and let $\BV^\eps$ be the Vorono\u{\i} tesselation associated with
 $\BS^\eps$, that is the set $\BV^\eps:= \set{V_p^\eps : p \in \BS^\eps}$, with
 \begin{equation}
   V_p^\eps := \set{x \in M : \dist(x,p) \le \dist(x,q) \text{ for all } q \in
   \BS^\eps}.
 \end{equation}
 We note that for $\eps < \eps_0$ and  every $p \in \BS^\eps$, $V_p^\eps$ is a domain with piecewise smooth
 boundary, and that
 \begin{equation}
   \Vol_g(V_p^\eps) \asymp_M \eps^d.
 \end{equation}
Indeed, by maximality of the $\eps$-separated set $\BS^\eps$ we have that
$B_{\eps/2}(p) \subset V_p^\eps \subset B_{3\eps}(p)$. 
Let $\beta\in C^\infty(M)$ be a smooth positive function. For every $p \in
\BS^\eps$, let $r_{\eps,p} > 0$ be such that
 \begin{equation}
   \label{eq:holedef}
   \CH^{d-1}(\del B_{r_{\eps,p}}(p)) = \beta(p) \Vol_g(V_p^\eps).
 \end{equation}
 It also follows from \eqref{eq:areageoball} and \eqref{eq:boundarygeoball} that
 \begin{equation}
   r_{\eps,p}  \asymp_{M,\beta} \eps^{\frac{d}{d-1}}.
 \end{equation}
 Since the previous display holds uniformly for $p\in M$, we often abuse notation and
 write $r_\eps$ for $r_{\eps,p}$. We set 
 \begin{equation}
   \BT^\eps := \bigcup_{p \in \BS^\eps} B_{r_\eps}(p),
 \end{equation}
 $\Omega^\eps = M \setminus \BT^\eps$, and $Q_p^\eps = V_p^\eps \setminus
 B_{r_\eps}(p)$. See Figure \ref{pic:tesselation} for a depiction of this
 construction.

\begin{figure}\centering
  \pgfmathsetmacro{\xmin}{0.8}
\pgfmathsetmacro{\xmax}{5.3}
\pgfmathsetmacro{\ymin}{1.4} 
\pgfmathsetmacro{\ymax}{3.65}
\pgfmathsetmacro{\picscale}{\textwidth/1cm/(\xmax-\xmin)}
\begin{tikzpicture}[line cap=round,line join=round ,scale=\picscale]
\clip(\xmin,\ymin)rectangle(\xmax,\ymax);
\foreach[count=\i]\x/\y in {
99.0000 / 99.0000,
0.6096 / 60.8493,
5.9364 / 2.6704,
-0.4053 / 1.6624,
-27.1042 / 1.8572,
0.4487 / 2.2047,
1.3643 / 2.8791,
0.7112 / 44.5919,
1.5494 / 4.0190,
-0.7290 / 3.4274,
0.4282 / 4.0519,
0.8978 / 3.7231,
0.4976 / 2.8283,
0.8954 / 3.1083,
6.0167 / 0.8534,
3.4973 / 2.8372,
1.9175 / 3.2107,
1.8687 / 3.8060,
2.4685 / 2.8573,
2.9023 / 3.1604,
2.6259 / 14.6407,
3.5537 / 4.0497,
2.5221 / 4.0184,
2.9410 / 3.7241,
4.0463 / 3.1827,
4.4999 / 2.8085,
5.1955 / 3.0527,
4.0751 / 3.7562,
4.4295 / 2.1858,
4.4991 / 1.0416,
5.0343 / 1.9343,
5.0993 / 1.3502,
1.5437 / -40.5229,
2.4068 / 2.3871,
4.3855 / 3.9138,
5.0113 / 3.6916,
4.1768 / 2.0309,
4.0676 / 1.5091,
3.5147 / 2.3445,
3.0435 / 1.9362,
4.4199 / 0.6617,
3.9331 / 0.3971,
4.9018 / 0.2668,
5.9855 / -33.3508,
4.4396 / -13.2896,
0.9318 / 0.1271,
0.4662 / 0.4677,
0.4691 / 1.0128,
1.9548 / 0.3359,
1.6063 / 0.5361,
3.4034 / 1.2891,
3.0622 / 1.5700,
2.5110 / 1.2524,
3.4369 / 0.6221,
2.8931 / 0.3118,
2.5141 / 0.4982,
2.0132 / 1.4753,
1.0290 / 1.9391,
1.1280 / 1.3263,
1.4891 / 1.1261,
1.4955 / 2.2936,
2.0091 / 2.1308 
}{\coordinate(\i)at(\x,\y);}
\fill[cyan!20](40)--(34)--(19)--(20)--(16)--(39)--cycle;
\draw(50)--(46)--(33)--(49)--cycle;
\draw(55)--(45)--(42)--(54)--cycle;
\draw(45)--(42)--(41)--(43)--(44)--cycle;
\draw(60)--(50)--(46)--(47)--(48)--(59)--cycle;
\draw(60)--(50)--(49)--(56)--(53)--(57)--cycle;
\draw(56)--(53)--(52)--(51)--(54)--(55)--cycle;
\draw(54)--(42)--(41)--(30)--(38)--(51)--cycle;
\draw(43)--(15)--(32)--(30)--(41)--cycle;
\draw(59)--(48)--(4)--(6)--(58)--cycle;
\draw(62)--(57)--(60)--(59)--(58)--(61)--cycle;
\draw(62)--(34)--(40)--(52)--(53)--(57)--cycle;
\draw(52)--(40)--(39)--(37)--(38)--(51)--cycle;
\draw(38)--(30)--(32)--(31)--(29)--(37)--cycle;
\draw(13)--(6)--(4)--(5)--(10)--cycle;
\draw(61)--(7)--(14)--(13)--(6)--(58)--cycle;
\draw(62)--(34)--(19)--(17)--(7)--(61)--cycle;
\draw(40)--(34)--(19)--(20)--(16)--(39)--cycle;
\draw(39)--(16)--(25)--(26)--(29)--(37)--cycle;
\draw(31)--(3)--(27)--(26)--(29)--cycle;
\draw(14)--(12)--(11)--(10)--(13)--cycle;
\draw(18)--(9)--(12)--(14)--(7)--(17)--cycle;
\draw(24)--(20)--(19)--(17)--(18)--(23)--cycle;
\draw(28)--(22)--(24)--(20)--(16)--(25)--cycle;
\draw(36)--(27)--(26)--(25)--(28)--(35)--cycle;
\draw(12)--(9)--(8)--(11)--cycle;
\draw(23)--(18)--(9)--(8)--(2)--(21)--cycle;
\draw(24)--(22)--(21)--(23)--cycle;
\begin{scope}[fill=black!5]
\filldraw(0.3748,-0.1144)node(p1){$\scriptstyle\bullet$};
\filldraw(1.4959,-0.0975)node(p2){$\scriptstyle\bullet$};
\filldraw(2.4049,-0.1067)node(p3){$\scriptstyle\bullet$};
\filldraw(3.4628,0.0136)node(p4){$\scriptstyle\bullet$};
\filldraw(4.4304,0.0494)node(p5){$\scriptstyle\bullet$};
\filldraw(5.3862,0.0802)node(p6){$\scriptstyle\bullet$};
\filldraw(-0.0582,0.7366)node(p7){$\scriptstyle\bullet$};
\filldraw(0.9933,0.7311)node(p8){$\scriptstyle\bullet$}circle(0.1513);
\filldraw(2.0981,0.9507)node(p9){$\scriptstyle\bullet$}circle(0.1402);
\filldraw(2.9263,0.9540)node(p10){$\scriptstyle\bullet$}circle(0.1413);
\filldraw(3.9117,1.0035)node(p11){$\scriptstyle\bullet$}circle(0.1384);
\filldraw(5.0223,0.7718)node(p12){$\scriptstyle\bullet$}circle(0.1508);
\filldraw(0.5812,1.5974)node(p13){$\scriptstyle\bullet$}circle(0.1711);
\filldraw(1.5617,1.7557)node(p14){$\scriptstyle\bullet$}circle(0.1344);
\filldraw(2.4611,1.7614)node(p15){$\scriptstyle\bullet$}circle(0.1342);
\filldraw(3.6406,1.8216)node(p16){$\scriptstyle\bullet$}circle(0.1295);
\filldraw(4.5841,1.6241)node(p17){$\scriptstyle\bullet$}circle(0.1290);
\filldraw(5.5417,1.7308)node(p18){$\scriptstyle\bullet$};
\filldraw(-0.0447,2.5829)node(p19){$\scriptstyle\bullet$};
\filldraw(0.9950,2.5014)node(p20){$\scriptstyle\bullet$}circle(0.1364);
\filldraw(1.8595,2.6951)node(p21){$\scriptstyle\bullet$}circle(0.1306);
\filldraw(3.0150,2.5436)node(p22){$\scriptstyle\bullet$}circle(0.1451);
\filldraw(3.9991,2.5783)node(p23){$\scriptstyle\bullet$}circle(0.1288);
\filldraw(4.9367,2.4722)node(p24){$\scriptstyle\bullet$}circle(0.1610);
\filldraw(0.3577,3.4067)node(p25){$\scriptstyle\bullet$}circle(0.1795);
\filldraw(1.4354,3.4025)node(p26){$\scriptstyle\bullet$}circle(0.1373);
\filldraw(2.3620,3.4785)node(p27){$\scriptstyle\bullet$}circle(0.1427);
\filldraw(3.4811,3.4016)node(p28){$\scriptstyle\bullet$}circle(0.1612);
\filldraw(4.6306,3.3439)node(p29){$\scriptstyle\bullet$}circle(0.1423);
\filldraw(5.5186,3.5999)node(p30){$\scriptstyle\bullet$};
\filldraw(-0.1497,4.3470)node(p31){$\scriptstyle\bullet$};
\filldraw(1.0102,4.3389)node(p32){$\scriptstyle\bullet$};
\filldraw(2.0750,4.3609)node(p33){$\scriptstyle\bullet$};
\filldraw(2.9758,4.3521)node(p34){$\scriptstyle\bullet$};
\filldraw(4.0701,4.4480)node(p35){$\scriptstyle\bullet$};
\filldraw(4.9351,4.2015)node(p36){$\scriptstyle\bullet$};
\end{scope}

\draw[latex-latex,red](p14.center)--(p15.center)node[sloped,above,pos=0.65]{$>\varepsilon$};
\draw[latex-latex,red](p22.center)--($(p22)+(0:0.1451cm)$)
node[below left=-2pt]{$r_{\varepsilon}$}
node[above left=2em,cyan!66!black]{$Q_p^{\varepsilon}$}
;
\end{tikzpicture}
\caption{Vorono\u{\i} tesselation associated with a maximal $\eps$-separated subset}%
\label{pic:tesselation}%
\end{figure}
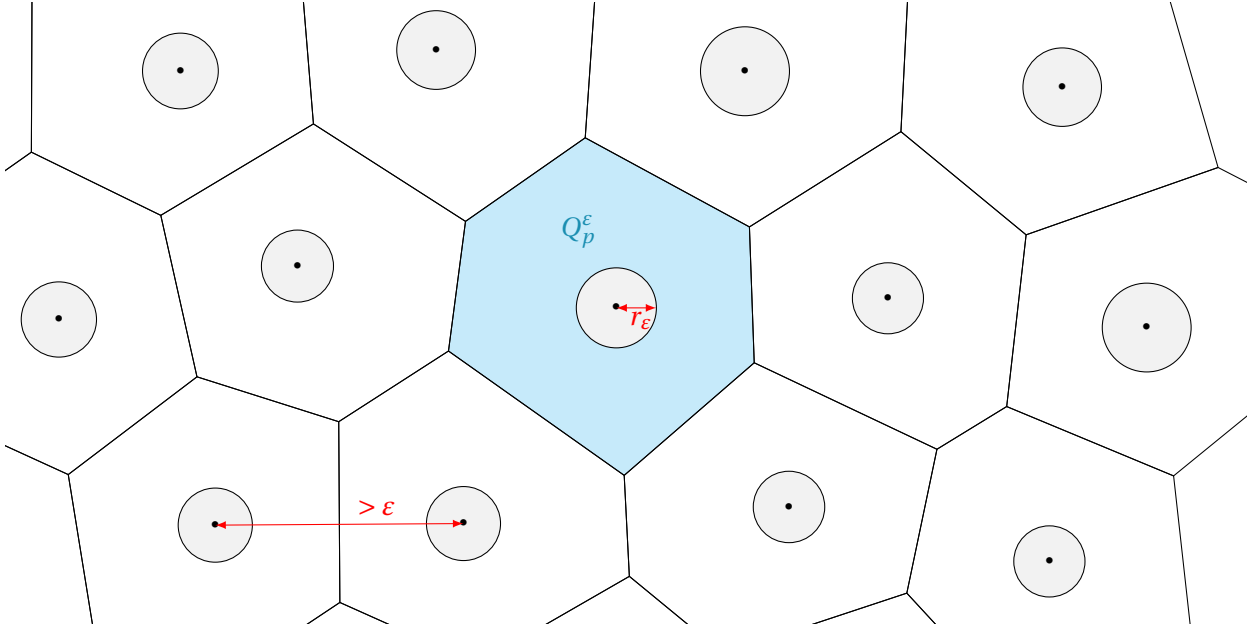

Furthermore, we have that $\dist(B_{r_{\eps,p}}(p),B_{r_{\eps,q}}(q)) \ge
\eps - \bigo[M,\beta]{\eps^{\frac{d}{d-1}}}$ for all $p \ne q \in \BS^\eps$. We see that by
construction, for every $0 < \eps < \eps_0$,
\begin{equation}
  \CH^{d-1}(\del \Omega^\eps) = \sum_{p\in\BS^\eps}\beta(p)\Vol_g(V_p^\eps).
\end{equation}
It also easy to see that the measure $\de A_g$ obtained on $M$ by
restriction of the
Hausdorff measure $\CH^{d-1}$ to $\del \Omega^\eps$ converges weak-$*$ to
the weighted Lebesgue measure $\beta \de \mu_g$ on $M$. That is, for each continuous function $f$ on $M$,
\begin{equation}
  \label{eq:weakstar}
\int_{\partial\Omega^\eps}f \de A_g \xrightarrow{\eps\to
0}\int_M f\,\beta\de \mu_g.
\end{equation}
This already addresses the first part of Theorem \ref{thm:approx}, and by
considering $f \equiv 1$ in \eqref{eq:weakstar} we see that
\begin{equation}
  \CH^{d-1}(\del\Omega^\eps) \xrightarrow{\eps \to 0} \int_M \beta \de \mu_g.
\end{equation}
We study the sequence of eigenvalue problems on $\Omega^\eps$
\begin{equation}
  \label{prob:steklovhom}
  \begin{cases}
    \Delta u^{(\eps)} = 0 & \text{in } \Omega^\eps; \\
    \del_\nu u^{(\eps)} = \sigma^{(\eps)} u^{(\eps)} & \text{on } \del
    \Omega^\eps,
  \end{cases}
\end{equation}
and for every eigenfunction $\ukn$, we define $\Ukn : M \to \R$ as the unique
function equal to $\ukn$ on
$\overline{\Omega^\eps}$ and harmonic in $\BT^\eps$.
The next theorem is the central technical result of this paper, and is in the
flavour of the main theorem of \cite{GHL}. It is also readily seen to imply
directly Theorem \ref{thm:approx} by providing the appropriate sequence
$\Omega^\eps$.

\begin{theorem}
  \label{thm:homo}
  The eigenvalues $\sigma_k^{(\eps)}$ of Problem \eqref{prob:steklovhom}
  converge as $\eps \to 0$ to the eigenvalue $\lambda_k(M,g,\beta)$ 
  defined in Problem \eqref{prob:laplace}. Up to choosing a subsequence, the
  extensions $\Ukn$ to $M$ of the eigenfunctions $\ukn$ converge
  weakly in $\RH^1(M)$ to the corresponding Laplace eigenfunction $f_k$ on $M$,
  where $f_k$ is normalised to $\RL^2(M,\beta \de \mu_g)$ norm $1$.
\end{theorem}

The proof is split in two main steps and is the subject of Section \ref{sec:homo}.

The first step is to show that there is a subsequence
$(\sigma_k^{(\eps)},\Ukn)$ converging to a weak solution $(\lambda,\phi)$ of the
weighted Laplace eigenvalue problem. In other words, the pair $(\lambda,\phi)$
satisfies
\begin{equation}
  \label{eq:weaklaplace}
  \forall v \in \RH^1(M), \qquad \int_{M} \nabla v \cdot \nabla \phi \de \mu_g = \lambda \int_M v \, \phi\,
  \beta \de \mu_g.
\end{equation}

The second step consists in proving that $(\lambda,\phi)$ has to be the $k$th
eigenpair of the weighted Laplace eigenvalue problem. This will be done by
showing that in the limit the functions $\Ukn$ do not lose any mass. Physically,
this can be interpreted as an instance of the Fermi exclusion principle, see.
e.g. the work of Colin de Verdi\`ere~\cite{cdv1986cmh} for an early application of
such an idea to create manifolds whose first Laplace eigenvalue have large
multiplicity.

\section{Analytic properties of perforated domains}
\label{sec:ineq}

In this section we describe analytic properties of the perforated manifolds
$\Omega^\eps$ and of the Vorono\u{\i} cells $V_p^\eps$. More precisely, we show 
that trace and extension operators are well behaved in
the homogenisation limiting process. 
We stress that many of the inequalities we show here would be
obviously satisfied for a fixed domain $\Omega$. However, the usual sufficient
conditions under which those inequalities would hold uniformly for a family of
domains
$\Omega^\eps$ either are not satisfied, or it is nontrivial to show
that they are indeed satisfied. We start by proving three lemmata about norms of
trace operators. We denote the annuli $A_p^\eps := B_{\eps/2}(p) \setminus
B_{r_\eps}(p)$.

The statement of Lemma \ref{lem:euctoriem} below is a generalisation of
\cite[Proposition 5.1]{bgt} for domains in a closed manifold. For domains
$\Omega \subset M$ whose
boundary is not necessarily piecewise smooth, we denote by $\operatorname{Per}(\Omega,M)$ their perimeter in $M$, which
corresponds to the Hausdorff measure of their reduced boundary $\del^* \Omega$.
Note that the topological boundary may in general be larger than the reduced
boundary.

\begin{lemma}
  \label{lem:euctoriem}
  Let $\set{\Omega_n \subset M: n \in \N}$ be a sequence of open, bounded
  domains such that $\CH^{d-1}(\del \Omega_n)$ is uniformly bounded.
  Assume that there exists
  $Q, \delta > 0$ such that for all $n\in\N$ and $x \in \del \Omega_n$,
  \begin{equation}
    \label{eq:characterisation}
    \sup\set{\frac{\CH^{d-1}(\del^*E \cap \del^*\Omega_n)}{\CH^{d-1}(\del^*E \cap
    \Omega_n)} : E \subset \Omega_n \cap B_\delta(x), \operatorname{Per}(E,\Omega_n) < \infty} <
    Q.
  \end{equation}
  Then, the trace operators $\tau_n : \RBV(\Omega_n) \to \RL^1(\del \Omega_n)$
  are bounded uniformly in $n$.
\end{lemma}

\begin{proof}
  For any $\eta > 0$, since $M$ is compact, we can choose $\delta$ small enough so that for every $x
  \in M$, the metric in geodesic polar coordinates in $B_{2\delta}(x)$ reads
  \begin{equation}
    g = \de \rho^2 + \rho^2 (1 + h(\rho,\theta)) \de \theta^2,
  \end{equation}
  with $\abs{h(\rho,\theta)} + \abs{\nabla h(\rho,\theta)} \le \delta^{1/2} \le
  \eta$. In other word, the diffeomorphism provided by the inverse of the exponential map, from $B_{2\delta}(x)$ to the ball of
  radius $2\delta$ in $\R^d$ is a $\RC^1$ $\eta$-perturbation of an isometry. For any $n$, the norms of
  $\RL^1(\del \Omega_n \cap B_\delta(x))$ and $\RBV(\Omega_n \cap B_\delta(x))$ change uniformly continuously on
  bounded sets under $\RC^1$ diffeomorphisms, and the same is true of the
  Hausdorff measures in \eqref{eq:characterisation}. By \cite[Proposition
  5.1]{bgt}, \eqref{eq:characterisation} implies that the trace operators are
  uniformly bounded on the pullbacks to the balls, and by the above discussion
  we can bring these estimates back to the manifold.
  \end{proof}

\begin{lemma}
  \label{lem:traceBVL1}
  The trace operators $\tau^\eps : \RBV(\Omega^\eps) \to
  \RL^1(\del\Omega^\eps)$ are
   bounded uniformly in $\eps$. 
\end{lemma}

\begin{proof}
  In order to apply Lemma \ref{lem:euctoriem}, we
  need to find $\delta, Q > 0$ such that for all $x \in \del \Omega^\eps$ and
  all $\eps > 0$, \eqref{eq:characterisation} holds. 
  A simple volume comparison yields
  that there is $c > 0$ such that for all $\delta > 0$ and
  $x \in \del \Omega^\eps$,
  \begin{equation}
    \label{eq:volcomp}
    \#\set{ p \in \BS^\eps : Q_p^\eps \cap B_\delta(x)\ne \varnothing} \le c\left( \frac
      \delta \eps
    \right)^d.
  \end{equation}
  Combining \eqref{eq:volcomp} with \eqref{eq:holedef}, for any $E \subset \Omega^\eps \cap B_\delta(x)$ of finite
  perimeter,
  \begin{equation}
    \CH^{d-1}(\del^* E \cap \del \Omega^\eps) \le \CH^{d-1}(\del\Omega^\eps \cap
    B_\delta(x)) \le C \delta^d,
  \end{equation}
  where $C$ depends on $M$, $g$ and $\beta$.
  We may then assume that the supremum is taken over sets $E$ such that
  \begin{equation}
  \CH^{d-1}(\del^* E \cap \Omega^\eps) \le C \delta^d,
\end{equation}
otherwise the ratio
  in \eqref{eq:characterisation} is bounded by $1$. Observe that
  \begin{equation}
    \CH^{d-1}(\del^* E \cap \del \Omega^\eps) = \sum_{p \in \BS^\eps}
    \CH^{d-1}(\del^* E \cap \del B_{r_\eps}(p)).
  \end{equation}

  For $p \in \BS^\eps$ and $t \in (0,\eps/4)$, define
  \begin{equation}
    F_{p,t} = E \cap \set{x : \dist(x,\del B_{r_\eps}(p)) \le t}.
  \end{equation}
  
  Assume that for some $t \in (0,\eps/4)$ we have that
  \begin{equation}
    \label{eq:localised}
    \CH^{d-1}(\del^* F_{p,t} \cap Q_p^\eps) \le 2\CH^{d-1}(\del^* E \cap Q_p^\eps).
  \end{equation}
  Without loss of generality, we have chosen $\delta$ small enough so that the retraction
  on a geodesic ball of radius $\delta' < \delta$ is a $2$-Lipschitz map
  uniformly for $x \in M$. This means that
  \begin{equation}
    \label{eq:locestimate}
    \begin{aligned}
      \CH^{d-1}(\del^* E \cap \del B_{r_\eps}(p)) &= \CH^{d-1}(\del^* F_{p,t} \cap \del
      B_{r_\eps}(p)) \\
      &\le 2 \CH^{d-1}(\del^* F_{p,t} \cap Q_p^\eps) \\
      &\le 4 \CH^{d-1}(\del^* E \cap Q_p^\eps).
  \end{aligned}
  \end{equation}
  Let $\tilde \BS^\eps = \set{p \in \BS^\eps : \eqref{eq:localised} \text{ does
  not hold}}$. If $\tilde \BS^\eps$ is empty, our claim holds since in that
  case
  \eqref{eq:locestimate} implies that 
  \eqref{eq:characterisation} holds with $Q = 4$. Let $p \in \tilde \BS^\eps$.
  
  Setting
  \begin{equation}
    h_p(t) := \CH^{d-1}\left(\del^*F_{p,t} \cap \set{x : \dist(x,\del
        B_{r_\eps}(p)) =
    t}\right),
    \end{equation}
    the coarea formula gives $\del_t \Vol_g(F_{p,t}) = h_p(t)$. It follows from the
    relative isoperimetric inequality \cite[Theorem 5.6.2]{evansgariepy} that
    there is a constant $c > 0$ depending on $M$ such that
    \begin{equation}
      \begin{aligned}
        c \Vol_g(F_{p,t})^{\frac{d-1}{d}} &\le \CH^{d-1}(\del^*F_{p,t} \cap Q_p^\eps) \\ 
        &\le 2 h_p(t),
      \end{aligned}
    \end{equation}
  where the second inequality follows from \eqref{eq:localised} not holding at
  $p$. Integrating, we therefore have that
  \begin{equation}
    \label{eq:lowerbound}
    \begin{aligned}
      2\Vol_g(F_{p,\eps/4}) &= \left(\int_0^{\eps/4}
      \frac{h_p(t)}{\Vol_g(F_{p,t})^{\frac{d-1}{d}}} \de t\right)^d \\
      &\gg_{M} C \eps^d \\
      &\gg_{M,\beta} C \CH^{d-1}(\del B_{r_{\eps}}(p)).
    \end{aligned}
  \end{equation}
  On the other hand, it follows from the isoperimetric inequality and equation
  \eqref{eq:locestimate} that
  \begin{equation}
    \label{eq:upperbound}
    \begin{aligned}
      \sum_{p \in \tilde \BS^\eps} \Vol_g(E \cap Q_p^\eps) &\ll_{M,\beta}\CH^{d-1}(\del^*
    E)^{\frac{d}{d-1}}  \\
    &\ll_{M,\beta}  \left(\CH^{d-1}(\del^*E \cap \Omega^\eps) + \CH^{d-1}(\del^* E \cap \del
    \Omega^\eps)\right)^{\frac{d}{d-1}}  \\
    & \ll_{M,\beta}  \Biggl(\CH^{d-1}(\del^* E \cap \Omega^\eps) + \sum_{p \in \tilde
    \BS^\eps} \CH^{d-1}(
B_{r_\eps}(p)
  )\Biggr)^\frac{d}{d-1}.
  \end{aligned}
  \end{equation}
  Summing over $p \in  \tilde \BS^\eps$ in \eqref{eq:lowerbound} and inserting
 in \eqref{eq:upperbound}, we obtain $C$ depending only on $M$ and $\beta$ such that 
\begin{equation}
  1 \le C \Biggl(\sum_{p \in \tilde \BS^\eps} \CH^{d-1}(
B_{r_\eps}(p)
  )\Biggr)^{\frac{1}{d-1}}
  \left(1 + 
    \frac{\CH^{d-1}(\del^*E \cap \Omega^\eps)}{\sum_{p \in \tilde \BS^\eps}
  \CH^{d-1}(\del B_{r_\eps}(p))}\right)^{\frac{d}{d-1}}.
\end{equation}
It follows from the weak-$*$ convergence in \eqref{eq:weakstar} that for small enough $\eps$,
\begin{equation}
\CH^{d-1}(\del \Omega^\eps \cap
B_\delta(x)) < 2 \max_p \beta(p) \Vol_g(B_\delta(x)).
\end{equation}
This means that
we can choose $\delta$ small enough, depending on $M$ and $\beta$ but not on $\eps$ so that
\begin{equation}
  C \Biggl(\sum_{p \in \tilde \BS^\eps} \CH^{d-1}(\del
  B_{r_\eps}(p))\Biggr)^{\frac{d}{d-1}} \le \frac{1}{4},
\end{equation}
which means that 
\begin{equation}
  \label{eq:notlocal}
  1 \le 
  \frac{\CH^{d-1}(\del^*E \cap \Omega^\eps)}{\sum_{p \in \tilde \BS^\eps}
  \CH^{d-1}(\del B_{r_\eps}(p))}.
\end{equation}
Combining estimates \eqref{eq:locestimate} with \eqref{eq:notlocal} gives us
that for $\eps$ small enough,
\begin{equation}
  \CH^{d-1}(\del^*E \cap \del\Omega^\eps) \le 4 \CH^{d-1}(\del^* E \cap
  \Omega^\eps),
\end{equation}
establishing our claim.
\end{proof}
The following lemma follows from the previous one rather directly, but we state
it explicitly for ease of reference.
\begin{lemma}
  \label{lem:traceH1L2}
  The Sobolev trace operators $\gamma^\eps : \RH^1(\Omega^\eps) \to
  \RL^2(\del\Omega^\eps)$ are bounded uniformly in $\eps$. 
\end{lemma}

\begin{proof}
  Observe first that if $f \in \RH^1(\Omega^\eps)$, then $f^2 \in
  \RBV(\Omega^\eps)$. Indeed, $\norm{f}_{\RL^2(\Omega^\eps)}^2 =
  \norm{f^2}_{\RL^1(\Omega^\eps)}$, and
  \begin{equation}
    \begin{aligned}
      \int_{\Omega^\eps} \abs{\nabla f^2} \de \mu_g &= \int_{\Omega^\eps} 2\abs{f
      \nabla f} \de \mu_g \\
      &\le \int_{\Omega^\eps} f^2 + \abs{\nabla f}^2 \de \mu_g. \\
  \end{aligned}
  \end{equation}
  We therefore have
  \begin{equation}
\begin{aligned}
  \norm{f}_{\RL^2(\del \Omega^\eps)}^2 &=
  \norm{f^2}_{\RL^1(\del\Omega^\eps)} \\ 
  &\le C \norm{f^2}_{\RBV(\Omega^\eps)} \\
  &\le 2 C \norm{f}_{\RH^1(\Omega^\eps)}^2,
\end{aligned}
  \end{equation}
  proving our claim.
\end{proof}

The next Lemma describes the behaviour of the operator of harmonic extension
inside the holes $\BT^\eps$.
\begin{lemma}
  \label{lem:harmextbounded}
  The harmonic extension operator $h^\eps : \RH^1(\Omega^\eps) \to
  \RH^1(\BT^\eps)$ has norm uniformly bounded in $\eps$. Furthermore,
  \begin{equation}
    \norm{h^\eps}_{\RH^1(\Omega^\eps)\to \RL^2(\BT^\eps)} \xrightarrow{\eps\to 0}0. 
  \end{equation}
\end{lemma}

\begin{proof}
  It is clearly sufficient to show that for $\delta$ small enough and $p \in \Omega$, the harmonic extension
  operator $\RH^1(B_{2\delta}(p) \setminus B_{\delta}(p)) \to
  \RH^1(B_{\delta}(p))$ is bounded uniformly in $\delta$, and that the $\RL^2$ norm of the
  harmonic extension in $B_\delta(p)$ goes to $0$. This follows directly from
  \cite[Example 1, p. 40]{RauchTaylor}, where this is shown in the Euclidean setting and the observation that for small enough
  $\delta$, c.f. equation \eqref{eq:areageoball}, geodesic balls and spherical
  shells are
  mapped to Euclidean balls and spherical shells by $\RC^1$-small
  perturbations of an isometry, and that all quantities involved are uniformly
  continuous in such perturbations.
\end{proof}
Finally, we will require that the Poincar\'e--Wirtinger
inequality of the perforated Vorono\u{\i} cells $Q_p^\eps$ hold uniformly in both $p \in M$
and $\eps > 0$. To this end, for any domain $\Omega \subset M$, denote by $\mu_1(\Omega)$ the first
non-trivial Neumann eigenvalue of $\Omega$, and for any $f : U \subset M \to \R$, 
\begin{equation}
  m_f := \int_U f \de \mu_g.
\end{equation}

 \begin{lemma}
  \label{lem:cheeg}
  There is $c,\eps_0 > 0$ depending only on $M$ and $\beta$ such that for $0 < \eps < \eps_0$, $p \in
  \BS^\eps$, and all $f \in \RH^1(Q_p^\eps)$
  \begin{equation}
    \int_{Q_p^\eps} \abs{f - m_f}^2 \de \mu_g \le c\eps^{2} \int_{Q_p^\eps}
    \abs{\nabla f}^2 \de \mu_g.
  \end{equation}
\end{lemma}

\begin{proof}
  It follows from the variational characterisation for Neumann eigenvalues that 
  \begin{equation}
    \int_{Q_p^\eps} \abs{f - m_f}^2 \de \mu_g \le \frac{1}{\mu_1(Q_p^\eps)} \int_{Q_p^\eps}
    \abs{\nabla f}^2 \de \mu_g
  \end{equation}
  so that it is equivalent to show that for all $0 < \eps < \eps_0$,
  \begin{equation}
    \mu_1(Q_p^\eps) \ge c^{-1} \eps^{-2}
  \end{equation}
  for some $ c > 0$.

  Since the Vorono\u{\i} cells $V_p^\eps$ are geodesically convex and have
  diameter $\diam(V_p^\eps) = \bigo{\eps}$, uniformly in $p \in \BS^\eps$, it
  follows from \cite[Theorem 1.2]{hkp} that there is a constant $C$ depending
  only on the curvature and dimension of $M$ such that
  \begin{equation}
    \label{eq:lowerboundmu1}
    \mu_1(V_p^\eps) \ge C \eps^{-2}.
  \end{equation}
  Let $w$ be the first non-constant Neumann eigenfunction of $Q_p^\eps$,
  normalised to $\norm{w}_{\RL^2(Q_p^\eps)} = 1$, and let
  $\hat w$ be the function defined on $V_p^\eps$ as the harmonic extension
  to $
B_{r_\eps}(p)
  $, i.e. as
  \begin{equation}
    \hat w(x) = \begin{cases}
      w(x) &\text{if } x \in Q_p^\eps \\
      h^\eps w(x) & \text{if } x \in
B_{r_\eps}(p)
      ,
    \end{cases}
  \end{equation}
  where $h^\eps$ is defined in Lemma \ref{lem:harmextbounded}. It follows from the Cauchy--Schwarz inequality and Lemma \ref{lem:harmextbounded} that
\begin{equation}
  m_w := \int_{V_p^\eps} \hat w(x) \de \mu_g = \int_{ 
  B_{r_\eps}(p)}h^\eps w(x) \de
  A_g =  \smallo{\eps^{\frac{d^2}{2(d-1)}}}.
\end{equation}
Using $\hat w - m_w$ as a test function for the first Neumann eigenvalue in
$V_p^\eps$ we have from Lemma \ref{lem:harmextbounded} that there is a constant
$c$ such that
\begin{equation}
  \begin{aligned}
    \mu_1(Q_p^\eps) &= \int_{Q_p^\eps} \abs{\nabla (\hat w - m_w)}^2 \de \mu_g \\
    &\ge c  \int_{V_p^\eps} \abs{\nabla (\hat w - m_w)}^2 \de \mu_g \\
    &\ge c \mu_1(V_p^\eps)\norm{\hat w - m_w}^2_{\RL^2(V_p^\eps)} \\
    &\ge c \eps^{-2} (1 + \smallo{1}),
  \end{aligned}
\end{equation}
concluding the proof.
\end{proof}

\section{Analytic properties of Steklov eigenpairs}
\label{sec:stekprop}
In this section, we obtain analytic properties of the Steklov eigenvalues
$\sigma_k^{(\eps)}:= \sigma_k(\Omega^\eps)$, and of Steklov eigenfunctions
$\ukn$. We
start by obtaining bounds on $\sigma_k^{(\eps)}$ which are uniform in $\eps$.

\begin{lemma}
  \label{lem:boundstek}
  For all $k\in \N$, $\beta \in C^\infty(M)$, we have as $\eps \to 0$
  \begin{equation}
    \sigma_k^{(\eps)} := \sigma_k(\Omega^\eps) \le \lambda_k(M,g,\beta) +
     o_{M,k,\beta}(1).
  \end{equation}
\end{lemma}

\begin{proof}
  It is clearly sufficient to prove this statement for $\eps < \eps_0$ small enough. It
  follows from the variational characterisation of Steklov eigenvalues that
  \begin{equation}
    \sigma_k^{(\eps)} = \min_{\substack{E \subset \RL^2(\del \Omega^\eps) \\
    \dim(E) = k+1}} \max_{u \in E} \frac{\int_{\Omega^\eps} \abs{\nabla u}^2 \de
    \mu_g}{\int_{\del \Omega^\eps} u^2 \de x}.
  \end{equation}
  Let $f_0,\dotsc,f_k$ be the first $k+1$ normalised eigenfunctions of the
  weighted Laplacian on
  $M$. They are pairwise $\RL^2(M,\beta \de \mu_g)$ orthogonal, and since the
  $(d-1)$-dimensional Hausdorff measure restricted to $\del \Omega^\eps$
  converges weak-$*$ to $\beta\de \mu_g$, for $\eps$ small enough
  they span a $k+1$ dimensional subspace of $\RL^2(\del\Omega^\eps)$, and for
  $0 \le j \le k$,
  \begin{equation}
    \begin{aligned}
      \norm{f_j}_{\RL^2(\del \Omega^\eps)}^2 &= \int_{M}f_j^2(x)\beta(x)\de
      \mu_g + o_{M,k,\beta}(1).
    \end{aligned}
  \end{equation}
  Therefore, using $E = \operatorname{span}(f_0,\dotsc,f_k)$ as a test subspace for
  $\sigma_k^\eps$ yields
  \begin{equation}
    \begin{aligned}
      \sigma_k^{(\eps)} &\le \max_{f \in E} \frac{\int_{\Omega^\eps} \abs{\nabla
      f}^2 \de \mu_g}{\int_{\del \Omega^\eps} f^2 \de \mu_g} \\
      &\le \lambda_k(M,g,\beta) + o_{M,k,\beta}(1),      
  \end{aligned}
  \end{equation}
  which is what we set out to prove.
\end{proof}

We turn to the boundedness of the sequence $\set{\ukn}$ in
$\RL^\infty(\Omega^\eps)$.

\begin{lemma}
  \label{lem:linftybound}
  There is a $\eps_0,C > 0$ depending only on $k, \beta, M$ such that for all $0
  < \eps < \eps_0$,
  \begin{equation}
    \norm{\ukn}_{\RL^\infty(\Omega^\eps)} \le C
  \end{equation}
\end{lemma}

\begin{proof}
  It is shown in \cite[Theorem 3.1]{bgt} that for any Steklov eigenfunction $u$
  with eigenvalue $\sigma$ on a domain $\Omega$, 
  \begin{equation}
    \norm{u}_{\RL^\infty(\Omega)} \le C \norm{u}_{\RL^2(\del \Omega)},
  \end{equation}
  with $C$ depending polynomially only on $\sigma$, $\Vol_g(\Omega)$ and the
  norm of the trace operator $\tau : \RBV(\Omega) \to \RL^1(\del \Omega)$. Note
  that they only prove this statement for domains in $\R^d$, however a close
  inspection of their proof reveals that geometric dependence appears in only
  two places. The first one is on the norm of the extension operator from
  $\RBV(\Omega) \to \RBV(M)$, which depends only on the norm of $\tau$ (see
  \cite{evansgariepy}[Theorem 5.4.1]), and therefore is already accounted for.
  The second one is on the norm of the Sobolev embedding $\RBV(M) \to
  \RL^{\frac{d}{d-1}}(M)$, whose norm depends only on the
  Gagliardo--Nirenberg--Sobolev inequality, which chnages by at most a constant
  for $M$ compact.

  Lemma \ref{lem:traceBVL1} gives a uniform bound for $\norm{\tau^\eps}$, Lemma
  \ref{lem:boundstek} gives us a uniform bound for $\sigma_k^{(\eps)}$ while
  $\Vol_g(\Omega)$ is obviously bounded by $\Vol_g(M)$ and $\ukn$ is normalised
  to $\|\ukn\|_{\RL^2(\del \Omega^\eps)} = 1$. Thus
  $\|\ukn\|_{\RL^\infty}$ is bounded, uniformly in $\eps$.
\end{proof}

\section{The homogenisation limit} \label{sec:homo}

In this section, we prove Theorem \ref{thm:homo}. While the general scheme of
the proof follows the general idea in \cite{GHL}, we cannot use any periodic
structure in order to define the auxiliary functions required to prove
convergence. The major difference with general homogenisation methods 
will be the definition of those auxiliary functions on a cell by cell basis in
such a way as to obtain the desired convergence.

Our first step is to show that there are converging subsequences. This is done
in the following lemma. Recall that $\ukn$ are the Steklov eigenfunctions on
$\Omega^\eps$ and $\Ukn$ their extension to $M$, harmonic in $\BT^\eps$.

\begin{lemma}
  \label{lem:boundh1}
  There is a subsequence of $\set{\Ukn}$, which we still label by $\eps$,
  converging weakly in $\RH^1(M)$.
\end{lemma}

\begin{proof}
  It suffices to show that the sequence $\set{\Ukn}$ is bounded in
  $\RH^1(\Omega)$ as $\eps \to 0$. By Lemma \ref{lem:harmextbounded}, we have that
  \begin{equation}
    \norm{\Ukn}_{\RH^1(M)} \ll_{M,\beta}  \norm{\ukn}_{\RH^1(\Omega^\eps)}
  \end{equation}
  On the other hand, we have that
  \begin{equation}
    \norm{\nabla \ukn}_{\RL^2(\Omega^\eps)^d}^2 \le \sigma_k^{(\eps)} \le
    \lambda_k(M,g,\beta) + \smallo[M,\beta,k]{1},
  \end{equation}
  where the last bound follows from Lemma~\ref{lem:boundstek}.
  Furthermore, it follows from Lemma \ref{lem:linftybound} that
    $$
    \norm{\ukn}_{\RL^2(\Omega^\eps)} \le \Vol_g(\Omega^\eps)^{1/2}
    \norm{\ukn}_{\RL^\infty(\Omega^\eps)} = \bigo[M,\beta]{1} .
  $$
  Combining all of this yields indeed that the sequence $\set{\Ukn}$ is
  uniformly bounded in $\RH^1(M)$, so that it has a subsequence weakly
  converging in $\RH^1(M)$.
\end{proof}

\begin{prop}
  Let $k \in \N$. As $\eps \to 0$, the pairs $(\Ukn, \sigma_k^{(\eps)})$
  converge to a pair $(f,\lambda)$, so that $f$ is an eigenfunction of the
  weighted Laplace problem on $M$ with eigenvalue $\lambda$, the convergence of $\Ukn$ being weak in $\RH^1$.
\end{prop}

\begin{proof}
  Denote by $(\phi,\lambda)$ the weak limit (up to a subsequence) of
  $(\Ukn,\sigma_k^{(\eps)})$, we now aim to show that they are weak solutions of the
  weighted Laplace eigenvalue problem on $M$, i.e. that they satisfy
  \eqref{eq:weaklaplace}.
  For a real valued $v \in \RH^1(M)$, we have, using the
  weak formulation of Problem \eqref{prob:steklovhom} that
  \begin{equation}
    \label{eq:weakform}
    \int_M \nabla \Ukn \cdot \nabla v \de \mu_g = \sigma_k^{(\eps)}
    \int_{\del \Omega^\eps} \Ukn v \de A_g +
    \int_{\BT^\eps} \nabla \Ukn \cdot \nabla v \de \mu_g.
  \end{equation}
  In order to be able to consider smooth test functions in this weak
  formulation, we need to ensure that the family of bounded linear functionals $\Phi^\eps \in \RH^1(M)^*$
  given by
  \begin{equation}
    \Phi^\eps(v) :=  \sigma_k^{(\eps)}
    \int_{\del \Omega^\eps} \Ukn v \de A_g.
  \end{equation}
  is bounded uniformly in $\eps < \eps_0$. It indeed is, since we know from Lemma \ref{lem:boundstek} that $\sigma_k^{(\eps)}$
is bounded as $\eps \to 0$, and we have
\begin{equation}
  \begin{aligned}
  \abs{  \int_{\del \Omega^\eps} \Ukn v \de A_g} &\le 
  \norm{\gamma^\eps}_{\RH^1(\Omega^\eps) \to
  \RL^2(\del \Omega^\eps)}^2 \norm{\Ukn}_{\RL^2(\del\Omega^\eps)}
  \norm{v}_{\RH^1(M)} \\
  &= 
  \norm{\gamma^\eps}_{\RH^1(\Omega^\eps) \to
  \RL^2(\del \Omega^\eps)}^2 \norm{v}_{\RH^1(M)}.
\end{aligned}
\end{equation}
We have shown in Lemma \ref{lem:traceH1L2} that $\norm{\gamma^\eps}$ was
 bounded uniformly in $\eps$ for $0 < \eps < \eps_0$. By the Banach-Steinhaus
 theorem, the family $\set{\Phi^\eps}$ is uniformly bounded. We may assume
from now on that in the weak formulation of Problem \eqref{prob:steklovhom}, we
consider only $v$ in a dense subspace of $\RH^1(M)$, in particular we assume
$v \in \RC^\infty(M)$. 

That the first term in \eqref{eq:weakform} converges follows from weak convergence of $\Ukn$. That the
last term in \eqref{eq:weakform} converges to $0$ follows from the
Cauchy--Schwarz inequality and the observation that since $v \in \RC^\infty(M)$,
\begin{equation}
  \int_{\BT^\eps} \abs{\nabla v}^2 \de \mu_g \le \max_{x \in M} \abs{\nabla v(x)}^2
  \Vol_g(\BT^\eps) \xrightarrow{\eps \to 0} 0.
\end{equation}

We now study the boundary term in \eqref{eq:weakform}.
  For every $p \in \BS^\eps$, define a function $\Psi_p^\eps : Q_p^\eps \to \R$ satisfying the weak
  variational problem
  \begin{equation}
    \forall v \in \RH^1(Q_p^\eps), \qquad \int_{Q_p^\eps} \nabla \Psi_p^\eps \cdot
    \nabla v \de \mu_g = - c_{\eps,p} \int_{Q_p^\eps} v \de \mu_g +  \int_{\del
    B_{r_\eps}(p)} v \de A_g
    \label{eq:weakpsi}
  \end{equation}
  Choosing $v \equiv 1$, we see that a necessary and sufficient condition for
  the existence of a solution (see \cite[Theorem 5.7.7]{TaylorI}) is that
    uniformly in $p$,
    \begin{equation}
      c_{\eps,p} = \frac{\CH^{d-1}(\del B_{r_\eps}(p))}{\Vol_g(Q_p^\eps)} =
      \beta(p) + \bigo[M,\beta]{\eps^{\frac{d}{d-1}}},
    \end{equation}
    and uniqueness is guaranteed by requiring that $\int_{Q_p^\eps}\Psi_p^\eps
    \de A_g = 0$.  
    The function $\Psi_p^\eps$ satisfies the differential equation
    \begin{equation}
      \begin{cases}
        \Delta \Psi_p^\eps = c_{\eps,p} & \text{in } Q_p^\eps \\
        \del_\nu \Psi_p^\eps = 1 & \text{on } \del B_{r_\eps}(p) \\
        \del_\nu \Psi_p^\eps = 0 & \text{on } \del V_p^\eps.
      \end{cases}
      \label{eq:strongPsi}
    \end{equation}
    We have that for all test functions $v$, 
\begin{equation}
  \int_{\del \Omega^\eps} \ukn v \de A_g = \sum_{p \in \BS^\eps}
  \int_{Q_p^\eps} \nabla
  \Psi_p^\eps \cdot
  \nabla(\ukn v) \de \mu_g + 
  \underbrace{\sum_{p \in \BS^\eps} c_{\eps,p} \int_{Q_p^\eps} \ukn v \de \mu_g}_{\to
   \int_M
  \phi v \beta\de \mu_g},
  \label{eq:extrep}
\end{equation}
where convergence of the last term comes from strong $\RL^2$ convergence of
$\set{\Ukn}$. We show that the other term converges to $0$. 
Applying the generalised Hölder inequality, we obtain
\begin{equation}
  \label{eq:genhold}
  \abs{\int_{Q_p^\eps}\nabla \Psi_p^\eps \cdot \nabla (\ukn v) \de \mu_g} \le
  \norm{v}_{\RC^1(Q_p^\eps)}\norm{\nabla\Psi_p^\eps}_{\RL^2(Q_p^\eps)^d}
  \norm{\Ukn}_{\RH^1(Q_p^\eps)}.
\end{equation}
 Since $v$ is smooth,
$\norm{v}_{\RC^1(M)}$ is bounded, and  a fortiori the restriction to
$Q_p^\eps$ is bounded as well.
By applying the variational characterisation of $\Psi_p^\eps$ to itself, we obtain
\begin{equation}
  \label{eq:nabpsin}
  \begin{aligned}
    \norm{\nabla\Psi_p^\eps}^2_{\RL^2(Q_p^\eps)^d} &=  \int_{\del
    B_{r_\eps}(p)} \Psi_p^\eps \de
  A_g  \\
  &\le  \norm{\gamma^\eps} \sqrt{\CH^{d-1}(\del B_{r_\eps}(p))}
  \norm{\Psi_p^\eps}_{\RH^1(Q_p^\eps)}.
  \end{aligned}
\end{equation}
By Lemma \ref{lem:traceH1L2}, $\norm{\gamma^\eps}$ is bounded. Since
$\Psi_p^\eps$ has average $0$ on $Q_p^\eps$, the Poincaré--Wirtinger
inequality tells us that
\begin{equation}
  \norm{\Psi_p^\eps}_{\RH^1(Q_p^\eps)} \le \Biggl( 1 +
    \frac{1}{\mu_1(Q_p^\eps)}
  \Biggr)^{1/2}\norm{\nabla \Psi_p^\eps}_{\RL^2(Q_p^\eps)^d}.
\end{equation}
By Lemma \ref{lem:cheeg}, $\mu_1(Q_p^\eps) \to \infty$ as $\eps \to 0$. This, along with
the fact that $\CH^{d-1}(\del B_{r_\eps}(p)) \asymp \eps^d$ tells us that
\begin{equation}
  \label{eq:boundpsi}
  \norm{\nabla \Psi_p^\eps}_{\RL^2(Q_p^\eps)^d} = \bigo{\eps^{d/2}}.
\end{equation}
Putting this estimate and \eqref{eq:genhold} into \eqref{eq:extrep} yields
\begin{equation}
  \begin{aligned}
  \sum_{p \in \BS^\eps} \int_{Q_p^\eps} \nabla \Psi_p^\eps \cdot \nabla (\ukn v)
\de \mu_g &\le \sum_{p \in
\BS^\eps}\norm{\gamma^\eps}\norm{v}_{\RC^1(M)}\sqrt{\CH^{d-1}(\del B_{r_\eps}(p))}
  \norm{\Ukn}_{\RH^1(Q_p^\eps)}  \\
  &\ll_{M,\beta,v} \eps^{d/2} \norm{\Ukn}_{\RH^1(M)},
\end{aligned}
\end{equation}
which goes to $0$ as $\eps \to 0$. Therefore, in view of \eqref{eq:extrep} and
\eqref{eq:weakform}, we have that if $(\phi,\lambda)$ are the limits of
$(\Ukn,\sigma_k^\eps)$ they do indeed satisfy the weak variational problem
\begin{equation}
  \forall v \in \RH^1(M), \qquad \int_{M} \nabla \phi \cdot \nabla v \de \mu_g =
  \lambda  \int_M \phi v \beta\de \mu_g,
\end{equation}
in other word $\phi$ is a weak eigenfunction of the weighted Laplacian on $M$ with
eigenvalue $\lambda$.
\end{proof}

Now that we have established convergence to solutions of the limit problem, we
need the following lemma to show that there is no mass lost in the interior.

\begin{lemma}
  Let $\phi$ be the weak limit in $\RH^1$ of $\Ukn$. Then,
  \begin{equation}
  1 =  \lim_{\eps \to 0} \int_{\del \Omega^\eps} (\Ukn)^2 \de A_g = \int_M \phi^2
  \beta
  \de \mu_g.
  \end{equation}
  \label{lem:nolostmass}
\end{lemma}

\begin{proof}
  By considering $v = \ukn$ in equation \eqref{eq:extrep} we have that
  \begin{equation}
    \int_{\del \Omega^\eps} (\ukn)^2 \de A_g = \sum_{p \in \BS^\eps} \int_{Q_p^\eps}
  \nabla
  \Psi_p^\eps\cdot
  \nabla (\ukn)^2 \de \mu_g + 
  \underbrace{\sum_{p \in \BS^\eps} c_{\eps,p} \int_{Q_p^\eps} (\ukn)^2 \de
  \mu_g}_{\to
    \int_M \phi^2\beta
  \de \mu_g}.
  \end{equation}
Once again, we have to show that the other term converges to $0$ as $\eps \to
0$. 
From the generalised
Hölder inequality, we see that
\begin{equation}
  \begin{aligned}
    \int_{Q_p^\eps} \nabla \Psi_p^\eps \cdot \nabla (\ukn)^2 \de \mu_g &= 2
    \int_{Q_p^\eps} \ukn \nabla \Psi_p^\eps \cdot \nabla \ukn \de \mu_g \\
    &\le 2\norm{\nabla\Psi_p^\eps}_{\RL^2(Q_p^\eps)^2}
    \norm{\ukn}_{\RL^\infty(Q_p^\eps)}\norm{\nabla \ukn}_{\RL^2(Q_p^\eps)^2}.
\end{aligned}
\end{equation}
It follows from Lemma \ref{lem:linftybound} that
$\norm{\ukn}_{\RL^\infty(Q_p^\eps)}$ is bounded, uniformly in $\eps$. 
Furthermore, it follows from equation \eqref{eq:boundpsi}
that $\norm{\nabla \Psi_p^\eps}_{\RL^2(Q_p^\eps)^d} \ll \eps^{d/2}$, so that
\begin{equation}
  \begin{aligned}
    \sum_{p \in \BS^\eps} \int_{Q_p^\eps} \nabla \Psi_p^\eps \cdot \nabla(\ukn)^2
    \de \mu_g & \le \sum_{p \in \BS^\eps} C \eps^{d/2} \norm{\nabla
      \ukn}_{\RL^2(Q_p^\eps)^d} \\
      &\le C \eps^{d/2} \norm{\nabla \ukn}_{\RL^2(\Omega^\eps)^d},
\end{aligned}
\end{equation}
which goes to $0$ as $\eps \to 0$, thereby finishing the proof.
\end{proof}

\begin{proof}[Proof of Theorem \ref{thm:homo}]
  We first prove that all the eigenvalues converge, proceeding by induction on
  the rank $k$. The base case $k = 0$ is trivial : indeed, the eigenvalue
  $\sigma_0^{(\eps)}$ obviously converges to $\lambda_0 = 0$, and the normalised
  constant eigenfunctions of each problem satisfy by construction
  \begin{equation}
    \begin{aligned}
      U_0^{(\eps)}(x) &= \CH^{d-1}(\del \Omega^\eps)^{-1/2} \\
      &\xrightarrow{\eps \to 0} \left(\int_{M} \beta \de \mu_g\right)^{-1/2}
      \\
      &= \phi_0(x)
\end{aligned}
\end{equation}
  Suppose now that for all $0 \le j \le k-1$, $U_j^{(\eps)}$ converges to $\phi_j$
  weakly in $\RH^1(\Omega)$. 
  We have already shown in Lemma \ref{lem:boundstek} that for all $k$,
  $\sigma_k^{(\eps)} \le \lambda_k(M,g,\beta) + \smallo 1$. We now show that the
  eigenvalues $\lambda_k(M,g,\beta)$ are
  bounded above by $\sigma_k^{(\eps)} + \smallo 1$. Suppose that the limit
eigenpair for $(\sigma_k^{(\eps)},\ukn)$ is  $(\lambda_j,\phi_j)$
for some $0 \le j \le k-1$. We have
  that
  \begin{equation}
    \begin{aligned}
    0 &= \lim_{\eps \to 0} \int_{\del \Omega^\eps} u_k^{(\eps)} u_{j}^{(\eps)}
    \de A_g \\
    &= \lim_{\eps \to 0} \int_{\del \Omega^{\eps}} u_k^{(\eps)} \phi_j \de A_g +
      \int_{\del \Omega^{\eps}} u_k^{(\eps)} (u_j^{(\eps)} - \phi_j) \de A_g.
  \end{aligned}
  \end{equation}
  The first term converges to $1$ by the assumption that
  \begin{equation}
    \int_{M} \phi_j^2 \beta \de \mu_g = 1.
  \end{equation}
   For the second term, Cauchy-Schwarz inequality and the
  normalisation of $\ukn$ tells us that
  \begin{equation}
    \int_{\del \Omega^\eps} \ukn(u_j^{(\eps)} - \phi_j) \de A_g \le \norm{u_j^{(\eps)} -
    f_j}_{\RL^2(\del \Omega^\eps)}.
  \end{equation}
  It follows from Lemma \ref{lem:nolostmass} that this limit converges to $0$,
resulting in a contradiction. This means that the eigenvalue $\lambda_k$ to
which $\sigma_k^{(\eps)}$ converges has a rank higher than $k-1$. Combining this
with the upper bound on $\lambda_k$ implies that $\sigma_k^{(\eps)}$ converges
indeed to $\lambda_k$. Weak convergence of the eigenfunctions therefore follows, up
to taking a subsequence when the eigenvalues are multiple.
\end{proof}

\section{Isoperimetric inequalities}
\label{sec:isoineq}

We are now in a position to prove Theorem \ref{thm:main}.

\begin{proof}[Proof of Theorem \ref{thm:main}]
Let $\delta > 0$ and $g_\delta$ be a metric on the surface $M$ such that
\begin{equation}
  \Lambda_k(M,g_\delta) \ge \Lambda_k^*(M) - \frac \delta 2.
\end{equation}

By taking $\beta = 1$ in Theorem \ref{thm:approx}, there is a family of domains
$\Omega^\eps \subset M$ such that for all $0 <\eps < \eps_0$, $\CH^1(\del
\Omega^\eps) = \Vol_g(M)$
and such that $\sigma_k(\Omega^\eps,g_\delta) \to \lambda_k(M,g_\delta)$ as $\eps \to 0$. In other
words,
\begin{equation}
  \lim_{\eps \to 0} \Sigma_k(\Omega^\eps,g_\delta) = \Lambda_k(M,g_\delta),
\end{equation}
so that there is $\eps > 0$ such that $\Sigma_k(\Omega^\eps,g_\delta) \ge
\Lambda_k^*(M) -  \delta$. Since $\delta$ is arbitrary, we have that
\begin{equation}
  \Sigma_k^*(M) \ge \Lambda_k^*(M).
\end{equation}
for all $k \in \N$ and surfaces $M$.
\end{proof}

\subsection{\texorpdfstring{Lower bounds and exact values for
$\Sigma_k^*$}{Lower bounds and exact values for large Steklov eigenvalues}}
For any closed surface $M$ for which $\Lambda_k^*(M)$ is known,
Theorem~\ref{thm:main}, along with Corollary \ref{cor:equality} leads to an
exact value for $\Sigma_k^*$ when $k \in \set{1,2}$, whereas it yields lower
bounds when $k \ge 3$. 
We have already seen that $\Sigma_1^*(\S^2)=\Lambda_1^*(\S^2)=8\pi$ in
Corollary~\ref{cor:sphere}. More generally, it follows from
Karpukhin--Nadirashvili--Penskoi--Polterovich~\cite{KNPP} that
$$\Sigma_k^*(\S^2)\geq \Lambda_k^*(\S^2)= 8\pi k,$$
with equality when $k \le 2$. The supremum is saturated by a sequence of
Riemannian metrics degenerating to  $k$ kissing spheres of equal area.
It follows from Nadirashvili~\cite{nad96} that
$$\Sigma_1^*(\T^2)= \Lambda_1(\T^2)=\frac{8\pi^2}{\sqrt{3}}.$$
The maximizer is the equilateral flat torus.
For the orientable surface $M$ of genus two, it follows from
Nayatoni--Shoda \cite{NayaSho2019} that
$$\Sigma_1^*(M) = \Lambda_1^*(M)=16\pi.$$
Where the equality $\Lambda_1^*(M)=16\pi$ was initially conjectured in the paper~\cite{JLNNP} by
Jakobson--Levitin--Nadirashvili--Nigam--Polterovich. This time the
maximizer is realized by a singular conformal metric on the Bolza surface.
Some results are also known for non-orientable surfaces. For instance,
it follows from the work of Li--Yau~\cite{liyau} that for the projective plane,
$$\Sigma_1^*(\RP^2) = \Lambda_1^*(\RP^2)=12\pi,$$
where the maximal metric is the canonical Fubini--Study metric.
It follows from Nadirashvili--Penskoi \cite{nadpen} that
$$\Sigma_2(\RP^2) =\Lambda_2^*(\RP^2) = 20 \pi,$$
and from Karpukhin~\cite{karpRP2} that for all $k \ge 3$,
$$\Sigma_k(\RP^2) \ge \Lambda_k^*(\RP^2) = 4\pi(2 k + 1).$$
This time the maximal metric is achieved by a sequence of surfaces degenerating to a
union of a projective plane and $k-1$ spheres with their canonical metrics, the
ratio of the area of the projective planes to the area of the union of the
spheres being $3:2$.

Finally,  it follows from El Soufi--Giacomini--Jazar~\cite{EGJ} and Cianci--Karpukhin--Medvedev~\cite{CKM}
that
$$\Sigma_1(\mathbb{KL}) = \Lambda_1^*(\mathbb{KL})=12 \pi E\Biggl(\frac{2\sqrt 2}{3}\Biggr),$$
where $E$ is the complete elliptic integral of the second type.
The supremum for  is realized by a bipolar Lawson surface corresponding to the $\tau_{3,1}$-torus. The equality for $\Lambda_1^*$ was first conjectured by
Jakobson--Nadirashvili--Polterovich~\cite{JNP}.
      
There are also situations where lower bounds for $\Lambda_k^*$ can be transfered
to $\Sigma_k^*$. For instance, restricting to flat metrics on $\T^2$, it follows
from Kao--Lai--Osting~\cite{KLO} and Lagac\'e~\cite{lagace} that
\begin{equation}\label{eq:KLOL}
  \Lambda_k^*(\T^2)_{\text{flat}}:=
\sup_{\substack{g\in\CG(M)}g  \\ \text{ flat}} \Lambda_k(M,g)\geq\frac{4\pi^2 \ceil{\frac k 2}^2}{\sqrt{\ceil{\frac k 2}^2 - \frac 1 4}}
\end{equation}
and that $\Lambda_k^*$ is realised by a family of flat tori degenerating to a circle as $k \to \infty$.
It follows from Theorem~\ref{thm:main} that
$$\Sigma_k^*(\T^2)_{\text{flat}} := \sup_{\substack{g \text{ flat} \\ \Omega \subset M}} \Sigma_k(\Omega)\ge\frac{4\pi^2 \ceil{\frac k 2}^2}{\sqrt{\ceil{\frac k 2}^2 - \frac 1 4}}.$$
Note that it is also conjecture in~\cite{KLO} that~\eqref{eq:KLOL} is an equality. 
We record one last general result following from the same strategy.
\begin{cor}
  \label{cor:explicit}
$$\Sigma_k^*(M)\ge \Lambda_1^*(M) + 8\pi(k-1)$$
\end{cor}

\begin{proof}
  This follows from the work of Colbois--El Soufi \cite{colboisElSoufi}, see also
      \cite{KLO,KNPP} for further discussion, where it is shown that one can
      glue in appropriate ratios maximisers for the first eigenvalue in a
      topological class with spheres to obtain bounds on the $k$th normalised
      eigenvalue of the Laplacian.
\end{proof}

\section{First Steklov eigenvalue of free boundary minimal surfaces}
\label{sec:fbms}
In view of the proof of Theorem \ref{thm:tetrahedral}, we recall a few
definitions.

\begin{defi}\label{def:fundamentaldomain}
Let $G$ be a subgroup of the group of isometries of $\mathbb{B}^3$. 
A submanifold $\Omega\subset\mathbb{B}^3$ is called invariant under the action
of $G$ if $\psi(\Omega)=\Omega$ for all $\psi\in G$. 
Given $x\in\mathbb{B}^3$ we denote by $G(x)=\bigcup_{\psi\in G}\psi(x)$ the \emph{orbit} of $x$. 
A connected subset $W\subset\mathbb{B}^3$ is a \emph{fundamental domain}
for the action of $G$ on $\mathbb{B}^3$ if $G(x)\cap W$ contains exactly one
element for every $x\in\mathbb{B}^3$. 
Similarly, a connected subset $D\subset\Omega$ is called \emph{fundamental
domain} for $\Omega$ if $G(x)\cap D$ contains exactly one element for every
$x\in\Omega$.
\end{defi}

We fist prove the following lemma, concerning connectedness of subsets of
fundamental domains for reflection groups.

\begin{lemma}\label{lem:connectedness}
Let $G$ be a finite group generated by reflection along planes $\Pi_1,\dotsc,\Pi_n\subset\R^3$ passing through the origin such that $W\subset\B^3$ bounded by the planes $\Pi_1,\dotsc,\Pi_n$ and $\del\B^3$ is a fundamental domain for the action of $G$ on $\B^3$. 
Let $E \subset W$ be such that $G(E)$ is path connected. 
Then, $E$ is path connected.
\end{lemma}

\begin{proof}
By Definition \ref{def:fundamentaldomain}, every $x\in\B^3$ has a unique $\tilde{x}\in G(x)\cap W$. 
Moreover, it follows from the definition of $W$ that for every $x\in\B^3$ there exists $\delta_x>0$ such that 
\begin{align}\label{eqn:uniquechoice}
\set{\tilde{y}\;:\; y\in B_{\delta_x}(x)\cap\B^3}\subset B_{\delta_x}(\tilde{x})\cap W.
\end{align}
Let $x_0, x_1\in E$ be arbitrary and let $\gamma\colon[0,1]\to G(E)$ be a continuous path with $\gamma(0)=x_0$ and $\gamma(1)=x_1$. 
Let $\tilde\gamma\colon[0,1]\to E$ be given by $\tilde\gamma(t)=\widetilde{\gamma(t)}$. 
Since $E\subset W$, it is clear that $\tilde{\gamma}$ is well-defined satisfying $\tilde\gamma(0)=x_0$ and $\tilde\gamma(1)=x_1$.  
Moreover, \eqref{eqn:uniquechoice} implies that $\tilde{\gamma}$ is continuous, and thus connecting $x_0$ and $x_1$ in $E$.  
\end{proof}

The following Lemma states that the surfaces satisfying the hypotheses of
Theorem~\ref{thm:tetrahedral} have fundamental domains with the same structure
as those visualised in Figures \ref{fig:tetrahedral}, \ref{fig:octahedral} and \ref{fig:icosahedral}.

\begin{lemma}\label{lem:fundamentaldomain}
Let $\Omega\subset\mathbb{B}^3$ be an embedded free boundary
minimal surface of genus $0$ which has  
tetrahedral symmetry and $b=4$ boundary components or 
octahedral symmetry and $b\in\{6,8\}$ boundary components or 
icosahedral symmetry and $b\in\{12,20,32\}$ boundary components. 
Then $\Omega$ has a simply connected fundamental domain $D$ with piecewise smooth boundary $\partial D$. 
If $b=32$ then $\partial D$ consists of five edges and five right-angled corners. 
In the other cases, $\partial D$ has four edges and four corners, three of which are right-angled. 
\end{lemma}

\begin{proof}
The assumption that $\Omega$ has tetrahedral, octahedral or icosahedral symmetry
means that it is invariant under the action of the full symmetry group $G$ of a
certain platonic solid. Any such group is generated by reflections along planes through the origin. 
We can realise a fundamental domain $W$ for the action of $G$ on $\mathbb{B}^3$
as a four-sided wedge which is bounded by three symmetry planes $\Pi_1$,
$\Pi_2$, $\Pi_3$ of $\Omega$ and by $\partial\mathbb{B}^3$ as shown in
Figure~\ref{fig:tetrahedral} on the right. 
Indeed, given a platonic solid centred at the origin, let $v_1$ and $v_2$ be two
of its adjacent vertices, let $c_1=\frac{1}{2}(v_1+v_2)$ and let $c_2$ be the
center of a face adjacent to the edge between $v_1$ and $v_2$. 
Then, we can choose $\Pi_1$ as the plane through $v_1$, $v_2$ and the origin,
$\Pi_2$ as the plane through $v_1$, $c_1$ and the origin and $\Pi_3$ as the
plane through $c_1$, $c_2$ and the origin (see Figure \ref{fig:reflection}). 
In particular, $\Pi_1$ and $\Pi_3$ are orthogonal. See the classical book~\cite{Coxeter1973} for details on symmetries of platonic solids.

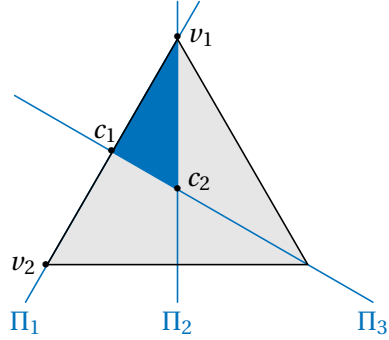
\begin{figure}
\pgfmathsetmacro{\plane}{1.5}
\pgfmathsetmacro{\size}{4cm}
\begin{tikzpicture}[line cap=round,line join=round,baseline={(0,0)},semithick]
\node[regular polygon, regular polygon sides=3, minimum size=\size,fill=black!10](P) at (0,0) {};
\fill[color={cmyk,1:magenta,0.5;cyan,1}](0,0)--(P.side 1)--(P.corner 1);
\pgfmathsetmacro{\planetop}{2.5}
\pgfmathsetmacro{\planebot}{-1.5}
\coordinate(A1)at({tan(30)*(\planebot-2)},{\planebot});
\coordinate(A2)at({tan(30)*(\planetop-2)},{\planetop});
\coordinate(B1)at({0*(\planebot-2)},{\planebot});
\coordinate(B2)at({0*(\planetop-2)},{\planetop});
\coordinate(C1)at({tan(-60)*(\planebot)},{\planebot});
\coordinate(C2)at({tan(-60)*(\planetop/2)},{\planetop/2}); 
\draw[color={cmyk,1:magenta,0.5;cyan,1}](A1)node[below]{$\Pi_1$}--(A2);
\draw[color={cmyk,1:magenta,0.5;cyan,1}](B1)node[below]{$\Pi_2$}--(B2);
\draw[color={cmyk,1:magenta,0.5;cyan,1}](C1)node[below]{$\Pi_3$}--(C2);
\draw(0,0)node{$\scriptstyle\bullet$}node[anchor=200]{$c_2$};
\draw(P.corner 1)node{$\scriptstyle\bullet$}node[right]{$v_1$};
\draw(P.corner 2)node{$\scriptstyle\bullet$}node[left]{$v_2$};
\draw(P.side 1)node{$\scriptstyle\bullet$}node[anchor=-70]{$c_1$};
\node[regular polygon, regular polygon sides=3, minimum size=\size,draw](P) at (0,0) {};
\end{tikzpicture} 
\caption{Reflection symmetries.}%
\label{fig:reflection}%
\end{figure}

The set $D\vcentcolon=W\cap\Omega$ is connected. 
This is a consequence of Lemma \ref{lem:connectedness} and the fact that
$\Omega$ is connected, being a free boundary minimal surface in the unit ball.
Moreover, $D$ meets $\partial W$ orthogonally. 
Along the planar faces of $W$, this follows from the assumption that $\Omega$ is
embedded and invariant under reflection and along $\partial\B^3\cap\partial W$
it is a direct consequence from the free boundary condition.
Hence, the curve $\partial D$ is piecewise smooth with corners where it meets
the edges of $W$. 
Moreover, the exterior angles along $\partial D$ are given by the angles between
the faces of $\partial W$. 
Let $\alpha_{1}\pi$ be the larger angle between $\Pi_1$ and $\Pi_2$ and let
$\alpha_{2}\pi$ be the larger angle between $\Pi_2$ and $\Pi_3$. 
All the other faces of $\partial W$ are pairwise orthogonal. 
Let $j,\ell_{1},\ell_{2}$ be the numbers of exterior angles along $\partial D$
with values $\frac{\pi}{2}$, $\alpha_{1}\pi$, $\alpha_{2}\pi$ respectively. 
By the argument above, these are all possible cases. 
We first observe that  
\begin{align*}
\int_\Omega K&=\abs{G}\int_{D} K, & 
\int_{\partial\Omega}\kappa&=\abs{G}\int_{\partial D}\kappa, 
\end{align*}
where we denote the Gauß curvature of a surface (here $\Omega$ or $D$) by $K$,
the geodesic curvature of its boundary by $\kappa$ and the number of elements in
the symmetry group $G$ by $\abs{G}$. 
By the Gauß--Bonnet theorem, we have the following formula for the Euler
characteristic $\chi(\Omega)$ of $\Omega$. 
\begin{align}\label{eqn:Gauss-Bonnet}
2\pi\chi(\Omega)  
&=\int_{\Omega}K+\int_{\partial \Omega}\kappa
=\abs{G}\biggl(\int_{D}K+\int_{\partial D}\kappa\biggr)
=\abs{G}\bigl(2\pi\chi(D)-j\tfrac{\pi}{2}-\ell_{1}\alpha_{1}\pi-\ell_{2}\alpha_{2}\pi\bigr).
\end{align}
Since $\Omega$ has genus $0$ and $b$ boundary components, $\chi(\Omega) = 2-b$
and equation \eqref{eqn:Gauss-Bonnet} yields
\begin{align}\label{eqn:20200228}
2\abs{G}\chi(D)&=\abs{G}\tfrac{j}{2}+\abs{G}\ell_{1}\alpha_{1}+\abs{G}\ell_{2}\alpha_{2}+2(2-b). 
\end{align}
In the case of tetrahedral symmetry we have $\abs{G}=24$ and $b=4$ as well as $\alpha_{1}=\alpha_{2}=\frac{2}{3}$. 
Simplifying equation \eqref{eqn:20200228}, we obtain 
\begin{align}\label{eqn:20200306-2251}
12\chi(D)&=3j+4(\ell_{1}+\ell_{2})-1. 
\end{align}
Any connected surface $D$ with boundary has Euler characteristic $\chi(D)\leq 1$. 
Since $j,\ell_{1},\ell_{2}$ must be nonnegative integers, the right hand side of equation \eqref{eqn:20200306-2251} is bounded from below by $-1$ and does not vanish which implies $\chi(D)=1$.  
Moreover, equation \eqref{eqn:20200306-2251} implies $j,\ell_{1},\ell_{2}\leq4$. 
By testing all combinations we obtain $j=3$ and $\ell_{1}+\ell_{2}=1$ as the only possibility. 
In particular, $D$ has $j+\ell_{1}+\ell_{2}=4$ corners and the topology of a disk as claimed.  

In the octahedral case, we have $\abs{G}=48$ and $b\in\{6,8\}$ as well as $\alpha_{1}=\frac{2}{3}$ and $\alpha_{2}=\frac{3}{4}$.   
In this case, equation \eqref{eqn:20200228} implies 
\begin{align*}
24\chi(D)&=6j+8\ell_{1}+9\ell_{2}-
\begin{cases}
2 &\text{ if $b=6$,} \\
3 &\text{ if $b=8$.}
\end{cases}
\end{align*}
As before, we conclude $\chi(D)=1$ and obtain $(j,\ell_{1},\ell_{2})=(3,1,0)$ if $b=6$ or $(j,\ell_{1},\ell_{2})=(3,0,1)$ if $b=8$.

With icosahedral symmetry, we have $\abs{G}=120$ and $b\in\{12,20,32\}$ as well as $\alpha_{1}=\frac{2}{3}$ and $\alpha_{2}=\frac{4}{5}$.   
Then, equation \eqref{eqn:20200228} implies 
\begin{align}\label{eqn:20200306-2201}
60\chi(D)&=15j+20\ell_{1}+36\ell_{2}-\begin{cases}
5 & \text{ if $b=12$,}\\
9 & \text{ if $b=20$,}\\
15& \text{ if $b=32$.}\\
\end{cases}
\end{align}
If $b\in\{12,20\}$ we obtain $\chi(D)=1$ and $(j,\ell_{1},\ell_{2})=(3,1,0)$ respectively $(j,\ell_{1},\ell_{2})=(3,0,1)$ as above. 
In the case $b=32$, equation \eqref{eqn:20200306-2201} has the solution $(j,\ell_{1},\ell_{2})=(1,0,0)$ with $\chi(D)=0$ which we need to exclude. 
Since the group order $\abs{G}=120$ exceeds the number $b=32$ of boundary components, there are no closed curves in $\partial D\cap\partial\mathbb{B}^3$. 
Consequently, and since $\Omega$ is embedded with boundary, $\partial D$ must have at least two corners on $\partial\mathbb{B}^3$ which implies $j\geq2$. 
In this case, the right hand side of \eqref{eqn:20200306-2201} is positive which implies $\chi(D)=1$. 
The equation simplifies to 
\begin{align*}
45&=15(j-2)+20\ell_1+36\ell_{2}
\end{align*}
and the only solution with integers $(j-2),\ell_{1},\ell_{2}\geq0$ is $(j,\ell_{1},\ell_{2})=(5,0,0)$. 
\end{proof}

We are now ready to prove our main result regarding free boundary minimal surfaces.
\begin{proof}[Proof of Theorem \ref{thm:tetrahedral}]
A result by McGrath \cite[Theorem 4.2]{McGrath2018} states $\sigma_1(\Omega)=1$ provided that $\Omega\subset\mathbb{B}^3$ is an embedded free boundary minimal surface which is invariant under a finite group $G$ of reflections satisfying the following two conditions. 
\begin{enumerate}
\item\label{condition:G1} The fundamental domain for the action of $G$ on $\mathbb{B}^3$ is a four-sided wedge $W$ bounded by three planes and $\partial\mathbb{B}^3$.
\item\label{condition:G2} The fundamental domain $D=W\cap\Omega$ for $\Omega$ is simply connected with boundary $\partial D$ which has at most five edges and intersects $\partial\Omega$ in a single connected curve.
\end{enumerate}

Let $D$ be the fundamental domain for $\Omega$ as given by Lemma~\ref{lem:fundamentaldomain}. 
Interpreting $D$ as free boundary minimal disk inside $W$, a result by Smyth
\cite[Lemma~1]{Smyth1984} states that the integral of the outward unit normal
vector field along $\partial D$ vanishes.
Consequently, $D$ meets all four faces of $W$ at least once. 
Hence, in the cases where $\partial D$ has exactly four edges, $\partial
D\cap\partial\Omega$ must be connected and \cite[Theorem 4.2]{McGrath2018}
applies.    

In the case $b=32$ where $\partial D$ has five edges and right angles, $\partial
D\cap\partial\Omega$ could be disconnected which would violate condition
\eqref{condition:G2}. 
We recall from the proof of Lemma~\ref{lem:fundamentaldomain} that the plane
$\Pi_2$ intersects $\Pi_1$ and $\Pi_3$ at angles different from $\frac{\pi}{2}$.

Since $\partial D$ has only right angles, it must avoid these two intersections
while still meeting the adjacent faces of $W$ (see Figure \ref{fig:icosahedral}
lower image). 
Hence, $\gamma=\partial D\cap\partial\Omega$ has indeed two connected components
$\gamma_1$ and $\gamma_2$. 
Let $e_i$ be the edge of $\partial D$ on $\Pi_i$ for $i\in\{1,2,3\}$ such that
in consecutive order 
\begin{align*}
\partial D&=e_1\cup\gamma_1\cup e_2\cup\gamma_2\cup e_3.
\end{align*}
In the following, we adapt McGrath's \cite{McGrath2018} approach to prove
$\sigma_1(\Omega)=1$ for the case at hand.  
Towards a contradiction, suppose that $\sigma_1(\Omega)<1$ and let $u$ be a
first eigenfunction for the Steklov eigenvalue problem satisfying 
\begin{align}\label{eqn:20202502}
\int_{\partial\Omega}u\,\de s&=0. 
\end{align}
Let $\CN=\{x\in\Omega\mid u(x)=0\}$ denote the \emph{nodal set} of $u$. 
As remarked in \cite{McGrath2018}, $\CN$ consists of finitely many arcs which intersect
in a finite set of points. 
By definition a \emph{nodal domain} of $u$ is a connected component of $\Omega\setminus\CN$. 
By Courant's nodal domain theorem, $u$ has exactly two nodal domains
$\CN^{\pm}\vcentcolon=\{x\in\Omega\mid \pm u(x)>0\}$, being a first non-trivial
eigenfunction.

We recall that the symmetry group $G$ of $\Omega$ is generated by reflections. 
Let $R\in G$ be any such reflection. 
According to \cite[Lemma 3.2]{McGrath2018} we have 
$u=\frac{1}{2}(u+u\circ R)$ since $\Omega$ is $R$-invariant with $\sigma_1(\Omega)<1$.

This implies that $u=u\circ \psi$ for any $\psi\in G$ which means that the two
nodal sets $\CN^{\pm}$ are invariant under the group action, i.\,e. they must
intersect every fundamental domain of $\Omega$ and still both be connected. 
Below we show that this contradicts the fact that the order of an element of the
icosahedral group is at most $10$.
 
Assumption \eqref{eqn:20202502} implies that $u$ restricted to $\gamma=\partial
D\cap\partial\Omega$ changes sign because being a Steklov eigenfunction, $u$
does not vanish on all of $\partial\Omega$. 
Consequently, an arc $\eta$ in $\CN$ either meets one connected component of
$\gamma$ or separates them by connecting two edges $e_i$ and $e_j$. 
In this case, at most ten alternating reflections on $\Pi_i$ and $\Pi_j$ close
up the curve $\eta$ and the enclosed region of $\Omega$ intersects at most ten
fundamental domains. 
However, $\Omega$ has $\abs{G}=120$ pairwise disjoint fundamental domains in total. 
This contradicts the fact that there are only two nodal domains which are
invariant under the group action. 
If the nodal line $\eta$ meets $\gamma_1$ or $\gamma_2$ then a similar
reflection argument shows that $u$ restricted to the corresponding connected
component of $\partial\Omega$ changes sign at least six times. 
Since $\Omega$ has genus~$0$, this implies that at least one of the two sets $\CN^{\pm}$ is disconnected which again contradicts \cite[Lemma 2.2]{McGrath2018}. 
This completes the proof.
\end{proof}

\appendix


\section{On the monotonicity of Steklov eigenvalues}
\label{sec:comment}

In this appendix, we elaborate on Remark \ref{rem:counter}, following
communication with Fraser and Schoen \cite{fsprivate}.
Given a compact orientable surface $\Omega$ of genus $\gamma$ with $b$ boundary components, we recall the notation from \eqref{eq:isostek} and set 
\begin{align*}
\sigma_1^*(\gamma,b)\vcentcolon=\sup_{g \in \CG(\Omega)}\Sigma_1(\Omega,g) 
\end{align*}
as in~\cite{fraschoen2}. The limit result \cite[Theorem 8.2]{fraschoen2} states that
$\sigma_1^*(0,b) \to 4 \pi$ as $b \to \infty$ and that the associated free
boundary minimal surfaces $\Omega_b$ converge to a double disk. In the proof, it
is shown that the area of $\Omega_b$ cannot concentrate near its boundary.
While this is true, a gap appears where this non-concentration phenomenon is used to deduce that all $\Omega_b$ must intersect a fixed smaller ball.
In~\cite{fraschoen2} this is used to show convergence of $\Omega_b$ to a non-trivial limit. There is another possibility: that the sequence of  maximisers $\Omega_b$ converge to the boundary $\S^2$. It is this
latter behaviour that is suggested by Theorem \ref{thm:approx} and Corollary
\ref{cor:sphere}, which leads us to state the following conjecture.

\begin{conjecture}
  There is a sequence $\set{\Omega_b : b \in \N} \subset \B^3$ of free boundary
  minimal surfaces of genus $0$ with $b$ boundary components which enjoys the
  following properties.
  \begin{enumerate}
    \item For every $b$, $\Omega_b$ maximises $\Sigma_1$ among surfaces of genus
      $0$ with $b$ boundary components.
    \item As $b \to \infty$, the measure on $\R^3$ obtained by restriction of
      the Hausdorff measure $\CH^1$ to $\del \Omega_b$ converges weak-$*$ to
      twice the measure obtained by restriction of $\CH^2$ on $\S^2$.
    \item As $b \to \infty$, $\Omega_b$ converges in the sense of varifolds to
      $\S^2$. 
  \end{enumerate}
  Furthermore, $\S^2$ is the unique limit point for $\set{\Omega_b}$ under the
  condition that they maximise $\Sigma_1$.
\end{conjecture}

We remark that this is not in contradiction with the existence of free boundary
minimal surfaces converging to the double disk as the number of boundary
components goes to infinity, it simply means that they are not global maximisers for
$\Sigma_1$. We also remark that a part of the gap in the proof of \cite[Theorem
8.2]{fraschoen2} appears also in the monotonicity result \cite[Proposition
4.3]{fraschoen2}, stating that $\sigma_1^*(\gamma,b) < \sigma_1^*(\gamma,b+1)$.
This was also mentioned to us in \cite{fsprivate}, along with a statement that
the result still holds and that a corrigendum is in preparation.

\bibliographystyle{plain}
\bibliography{homo}
\end{document}